\documentclass{llncs}

\usepackage{amssymb, graphics,enumerate, tkz-graph}
\usepackage[hypertexnames=false]{hyperref}
\usepackage[T1]{fontenc}
\usepackage[utf8]{inputenc}
\usepackage{microtype}

\usepackage{amsmath} 
\usepackage{amsfonts}

\newcommand{\ssi}{\subseteq_i}
\newcommand{\si}{\supseteq_i}

\newtheorem{oproblem}{Open Problem}

\newcounter{ctrclaim}[lemma]
\newcommand\displaycase[1]{{\em #1}}
\newcommand{\clm}[1]{\phantomsection\refstepcounter{ctrclaim}\noindent\displaycase{Claim \thectrclaim. }{#1}\\}
\newcounter{ctrcase}[lemma]
\newcommand{\thmcase}[1]{\medskip\phantomsection\refstepcounter{ctrcase}\noindent\displaycase{Case \thectrcase: }{\em #1}\\}

\title{Well-Quasi-Ordering versus Clique-Width:\\ New Results on Bigenic Classes\thanks{
Research supported by EPSRC (EP/K025090/1 and EP/L020408/1). An extended abstract of this paper appeared in the proceedings of IWOCA 2016~\cite{DLP16-conf}.}}

\author{Konrad K. Dabrowski\inst{1}, Vadim V. Lozin\inst{2} \and Dani\"el Paulusma\inst{1}}
\institute{School of Engineering and Computing Sciences, Durham University,\\ Science Laboratories, South Road, Durham DH1 3LE, United Kingdom
\texttt{\{konrad.dabrowski,daniel.paulusma\}@durham.ac.uk}
\and
Mathematics Institute, University of Warwick,\\
Coventry CV4 7AL, United Kingdom
\texttt{v.lozin@warwick.ac.uk}
}

\begin{document}
\maketitle

\begin{abstract}
Daligault, Rao and Thomass{\'e} asked whether a hereditary class of graphs well-quasi-ordered by the induced subgraph relation has bounded clique-width.
Lozin, Razgon and Zamaraev recently showed that this is not true for classes defined by infinitely many forbidden induced subgraphs.
However, in the case of finitely many forbidden induced subgraphs the question remains open and we conjecture that in this case the answer is positive. 
The conjecture is known to hold for classes of graphs defined by a single forbidden induced subgraph~$H$, 
as such graphs are well-quasi-ordered and are of bounded clique-width if and only if~$H$ is an induced subgraph of~$P_4$. 
For bigenic classes of graphs, i.e. ones defined by two forbidden induced subgraphs, there are several open cases in both classifications.
In the present paper we obtain a number of new results on well-quasi-orderability of bigenic classes, each of which supports the conjecture.
\end{abstract}

\section{Introduction}\label{s-intro}

Well-quasi-ordering is a highly desirable property and frequently discovered
concept in mathematics and theoretical computer science~\cite{FS01,Kruskal72}.
One of the most remarkable results in this area is Robertson and Seymour's proof of Wagner's conjecture,
which states that the set of all finite graphs is well-quasi-ordered by the minor relation~\cite{RS04-Wagner}.
One of the first steps towards this result was the proof of the fact that graph classes of bounded treewidth 
are well-quasi-ordered by the minor relation~\cite{RS90} (a graph
parameter~$\pi$ is said to be bounded for some graph class~${\cal G}$ if there exists a constant~$c$ such that 
$\pi(G)\leq c$ for all $G\in {\cal G}$).

The notion of clique-width generalizes that of treewidth in the sense that graph classes of bounded
treewidth have bounded clique-width, but not necessarily vice versa. The importance of both 
notions is due to the fact that many algorithmic problems that are NP-hard on general graphs
become polynomial-time solvable when restricted to graph classes of bounded treewidth or clique-width.
For treewidth this follows from the meta-theorem of Courcelle~\cite{Co92}, combined with a result of Bodlaender~\cite{Bo96}.
For clique-width this follows from combining results from several papers~\cite{CMR00,EGW01,KR03b,Ra07} with a result of Oum and Seymour~\cite{OS06}.

In the study of graph classes of bounded treewidth, we can restrict ourselves to minor-closed graph
classes, because from the definition of treewidth it immediately follows that the treewidth of a graph is never smaller than the treewidth of its minor.
This restriction, however, is not justified when we study graph classes of bounded clique-width, as the clique-width 
of a graph can be much smaller than the clique-width of its minor.
In particular, Courcelle~\cite{Co14} showed that if~${\cal G}$ is the class of graphs of clique-width~3 and~${\cal G}'$ is the class of graphs obtainable from graphs in~${\cal G}$ by applying one or more edge contraction operations, then~${\cal G}'$ has unbounded clique-width.
 On the other hand, the clique-width of a graph is 
never smaller than the clique-width of any of its induced subgraphs (see, for example,~\cite{CO00}). This allows us to restrict ourselves to classes of graphs
closed under taking induced subgraphs. Such graph classes are also known as {\em hereditary} classes. 

It is well-known (and not difficult to see) that a class of graphs is hereditary if and only if it can be characterized 
by a set of minimal forbidden induced subgraphs. Due to the minimality, the set~${\cal F}$ of forbidden induced subgraphs is always an antichain, 
that is, no graph in ${\cal F}$ is an induced subgraph of another graph in ${\cal F}$.
For some hereditary classes this set is finite, in which case we say that the class is
{\em finitely defined},
whereas for other hereditary classes
 (such as, for instance, bipartite graphs)
 the set of minimal forbidden induced subgraphs forms an infinite antichain.
The presence of these infinite antichains immediately shows that the induced subgraph relation is not a well-quasi-order.
In fact there even exist graph classes of bounded clique-width that are not well-quasi-ordered by the induced subgraph relation: take, for example, the class of cycles, which all have clique-width at most~4.
What about the inverse implication:
does well-quasi-ordering imply bounded clique-width?
This was stated as an open problem by Daligault, Rao and Thomass{\'e}~\cite{DRT10} and a negative answer to this question was recently given by
Lozin, Razgon and Zamaraev~\cite{LRZ15}. However, the latter authors disproved the conjecture by giving a hereditary class of graphs whose set of minimal forbidden induced subgraphs
is infinite. Hence, for finitely defined classes the question remains open
and we conjecture that in this case the answer is positive.

\begin{conjecture}\label{c-f}
If a finitely defined class of graphs ${\cal G}$ is well-quasi-ordered by the induced subgraph relation, then ${\cal G}$ has bounded clique-width.
\end{conjecture} 

\noindent
We emphasize that our motivation for verifying Conjecture~\ref{c-f} is not only mathematical but also algorithmic.
Should Conjecture~\ref{c-f} be true, then
for finitely defined classes of graphs  
the aforementioned algorithmic consequences of having bounded clique-width also hold for the property of being well-quasi-ordered by the induced subgraph relation. 

A class of graphs is {\it monogenic} or {\it $H$-free} if it is characterized by a single forbidden induced subgraph~$H$.
For monogenic classes, the conjecture is true. In this case, the two notions even coincide: a class of graphs defined by a single forbidden induced subgraph~$H$ is well-quasi-ordered if and only if it has bounded clique-width  if and only if~$H$ is an induced subgraph of~$P_4$
(see, for instance,~\cite{DP15,Da90,KL11}).

A class of graphs is {\it bigenic} or {\it $(H_1,H_2)$-free} if it is characterized by two
incomparable
forbidden induced subgraphs~$H_1$ and $H_2$.
The family of bigenic classes is more diverse than the family of monogenic classes. The questions of well-quasi-orderability and 
having bounded clique-width still need to be resolved.
Recently, considerable progress has been made towards answering the latter question for 
bigenic classes; see~\cite{DDP15} for the most recent survey, which shows that there are currently eight (non-equivalent) open cases.
With respect to well-quasi-orderability of bigenic classes, Korpelainen and Lozin~\cite{KL11} left all but 14 cases open.
Since then, Atminas and Lozin~\cite{AL15} proved that the class of $(K_3,P_6)$-free graphs is well-quasi-ordered by the induced subgraph relation and that the class of $(\overline{2P_1+P_2},P_6)$-free graphs is not, reducing the number of remaining open cases to~12. 
All available results for bigenic classes verify Conjecture~\ref{c-f}. Moreover, 
eight of the 12 open cases have bounded clique-width (and thus
immediately verify Conjecture~\ref{c-f}), leaving
four remaining open cases of bigenic classes for which we still need to verify Conjecture~\ref{c-f}.

\subsection*{Our Results}

Our first goal is to obtain more (bigenic) classes that are well-quasi-ordered by the induced subgraph relation 
and to support Conjecture~\ref{c-f} with further evidence. 
Our second and more general goal is to increase our general knowledge on well-quasi-ordered graph classes and the relation to the possible boundedness of their clique-width. 

\begin{sloppypar}
Towards our first goal we prove
in Section~\ref{s-yes} 
that the class of $(\overline{2P_1+P_2},\allowbreak P_2+\nobreak P_3)$-free graphs 
(which has bounded clique-width~\cite{DHP0}) 
is well-quasi-ordered by the induced subgraph relation. 
In Section~\ref{s-no} we
also determine, by giving infinite antichains, two bigenic classes that are not, namely the class of $(\overline{2P_1+P_2},P_2+\nobreak P_4)$-free graphs, which has unbounded clique-width~\cite{DHP0}, and
the class of $(\overline{P_1+P_4},P_1+\nobreak 2P_2)$-free graphs, for which boundedness of the clique-width is unknown
(see \figurename~\ref{fig:forbidden-graphs} for drawings of the five forbidden induced subgraphs).
Consequently, there are 
nine classes of $(H_1,H_2)$-free graphs for which we do not know whether they are well-quasi-ordered by the induced subgraph relation, and there are
two open cases left for the verification of Conjecture~\ref{c-f} for bigenic classes. 
We refer to Open Problems~\ref{o-wqo} and~\ref{o-con}, respectively, in Section~\ref{s-state} where we also give an exact description of the state-of-the-art for results on well-quasi-orderability and boundedness of clique-width for bigenic classes of graphs.
\end{sloppypar}

Towards our second goal, we aim to develop general techniques as opposed to tackling specific cases in an ad hoc fashion. 
Our starting point is a very fruitful technique used for determining (un)boundedness of the clique-width of a graph class~${\cal G}$. 
We transform a given graph from~${\cal G}$ via a number of elementary graph operations that do not modify the clique-width by ``too much'' into a graph from a class for which we do know whether or not its clique-width is bounded.

It is a natural question to research how the above modification technique can be used for well-quasi-orders. 
We do this in Section~\ref{s-permitted}. 
The permitted elementary graph operations that preserve (un)boundedness of the clique-width
are vertex deletion, subgraph complementation and bipartite complementation.
As we will explain in Section~\ref{s-permitted}, these three graph operations do not preserve well-quasi-orderability. We circumvent this
by investigating whether these three operations preserve boundedness of a graph parameter called uniformicity.
This parameter was introduced by Korpelainen and Lozin~\cite{KL11}, who proved that every graph class~${\cal G}$ of bounded uniformicity is well-quasi-ordered by the so-called labelled induced subgraph relation, which in turn implies that~${\cal G}$ is well-quasi-ordered by the induced subgraph relation. 
Korpelainen and Lozin~\cite{KL11} proved that boundedness of uniformicity is preserved by vertex deletion.
We prove that this also holds for the other two graph operations.

The above enables us to focus on boundedness of uniformicity.
However, we cannot always do this:
there exist graph classes of unbounded uniformicity that are well-quasi-ordered by the labelled induced subgraph relation.
As such, we sometimes need to rely only on the labelled induced subgraph relation directly.
Hence, in Section~\ref{s-permitted} we also show that the three permitted graph operations, 
vertex deletion, subgraph complementation and bipartite complementation, preserve well-quasi-orderability by the labelled induced subgraph relation. 

As explained in Section~\ref{s-state}, we believe that this graph modification technique will also be useful for proving well-quasi-orderability of other graph classes.
As such, we view the results in Section~\ref{s-permitted} as the second main contribution of our paper.

\begin{figure}
\begin{center}
\begin{tabular}{ccccc}
\begin{minipage}{0.18\textwidth}
\centering
\scalebox{0.7}{
{\begin{tikzpicture}[scale=1]
\GraphInit[vstyle=Simple]
\SetVertexSimple[MinSize=6pt]
\Vertices{circle}{a,b,d,e}
\Edges(e,a,b,e,d,b)
\end{tikzpicture}}}
\end{minipage}
&
\begin{minipage}{0.18\textwidth}
\centering
\scalebox{0.7}{
{\begin{tikzpicture}[scale=1,rotate=90]
\GraphInit[vstyle=Simple]
\SetVertexSimple[MinSize=6pt]
\Vertices{circle}{a,b,c,d,e}
\Edges(a,b,c,d,e,a)
\Edges(c,a,d)
\end{tikzpicture}}}
\end{minipage}
&
\begin{minipage}{0.18\textwidth}
\centering
\scalebox{0.7}{
{\begin{tikzpicture}[scale=1,rotate=90]
\GraphInit[vstyle=Simple]
\SetVertexSimple[MinSize=6pt]
\Vertices{circle}{a,b,c,d,e}
\Edges(b,c)
\Edges(d,e)
\end{tikzpicture}}}
\end{minipage}
&
\begin{minipage}{0.18\textwidth}
\centering
\scalebox{0.7}{
{\begin{tikzpicture}[scale=1,rotate=90]
\GraphInit[vstyle=Simple]
\SetVertexSimple[MinSize=6pt]
\Vertices{circle}{a,b,c,d,e}
\Edges(e,a,b)
\Edges(c,d)
\end{tikzpicture}}}
\end{minipage}
&
\begin{minipage}{0.18\textwidth}
\centering
\scalebox{0.7}{
{\begin{tikzpicture}[scale=1,rotate=180]
\GraphInit[vstyle=Simple]
\SetVertexSimple[MinSize=6pt]
\Vertices{circle}{a,b,c,d,e,f}
\Edges(b,c)
\Edges(d,e,f,a)
\end{tikzpicture}}}
\end{minipage}\\
\\
$\overline{2P_1+P_2}$ & $\overline{P_1+P_4}$ & $P_1+\nobreak 2P_2$ & $P_2+P_3$ & $P_2+\nobreak P_4$
\end{tabular}
\end{center}
\caption{The forbidden induced subgraphs considered in this paper.}
\label{fig:forbidden-graphs}
\end{figure}
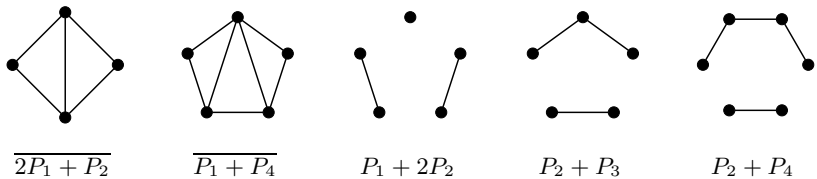

\section{Preliminaries}\label{s-prelim}

The {\em disjoint union} $(V(G)\cup V(H), E(G)\cup E(H))$ of two vertex-disjoint graphs~$G$ and~$H$ is denoted by~$G+\nobreak H$ and the disjoint union of~$r$ copies of a graph~$G$ is denoted by~$rG$. The {\em complement} of a graph~$G$, denoted by~$\overline{G}$, has vertex set $V(\overline{G})=\nobreak V(G)$ and an edge between two distinct vertices
if and only if these vertices are not adjacent in~$G$. 
For a subset $S\subseteq V(G)$, we let~$G[S]$ denote the subgraph of~$G$ {\em induced} by~$S$, which has vertex set~$S$ and edge set $\{uv\; |\; u,v\in S, uv\in E(G)\}$.
If $S=\{s_1,\ldots,s_r\}$ then, to simplify notation, we may also write $G[s_1,\ldots,s_r]$ instead of $G[\{s_1,\ldots,s_r\}]$.
We use $G \setminus S$ to denote the graph obtained from~$G$ by deleting every vertex in~$S$, i.e. $G \setminus S = G[V(G)\setminus S]$.
We write $H\ssi G$ to indicate that~$H$ is isomorphic to an induced subgraph of~$G$.

The graphs $C_r,K_r,K_{1,r-1}$ and~$P_r$ denote the cycle, complete graph, star and path on~$r$ vertices, respectively.
The graph~$K_{1,3}$ is also called the {\em claw}.
The graph~$S_{h,i,j}$, for $1\leq h\leq i\leq j$, denotes the {\em subdivided claw}, that is, the tree that has only one vertex~$x$ of degree~$3$ and exactly three leaves, which are of distance~$h$,~$i$ and~$j$ from~$x$, respectively. Observe that $S_{1,1,1}=K_{1,3}$.
We let ${\cal S}$ denote the class of graphs, each connected component of which is either a subdivided claw or a path.
For a set of graphs $\{H_1,\ldots,H_p\}$, a graph~$G$ is {\em $(H_1,\ldots,H_p)$-free} if it has no induced subgraph isomorphic to a graph in $\{H_1,\ldots,H_p\}$;
if~$p=1$, we may write $H_1$-free instead of $(H_1)$-free.

For a graph $G=(V,E)$ and a vertex $u \in V$, the set $N_G(u)=\{v\in V\; |\; uv\in E\}$ denotes the (open) {\em neighbourhood} of~$u$ in~$G$ and $N_G[u]=N_G(u) \cup \{u\}$ denotes the {\em closed neighbourhood} of~$u$.
We may write~$N(u)$ and~$N[u]$ instead of~$N_G(u)$ and~$N_G[u]$ if this is unambiguous.
A graph is {\em bipartite} if its vertex set can be partitioned into (at most) two independent sets.
The {\em biclique}~$K_{r,s}$ is the bipartite graph with sets in the partition of size~$r$ and~$s$ respectively, such that every vertex in one set is adjacent to every vertex in the other set.

Let~$X$ be a set of vertices of a graph $G=(V,E)$.
A vertex $y\in V\setminus X$ is {\em complete} to~$X$ if it is adjacent to every vertex of~$X$
and {\em anti-complete} to~$X$ if it is non-adjacent to every vertex of~$X$.
Similarly, a set of vertices $Y\subseteq V\setminus X$ is {\em complete} (resp. {\em anti-complete}) to~$X$ if every vertex in~$Y$ is complete (resp. anti-complete) to~$X$.
A vertex $y\in V\setminus X$
{\em distinguishes}~$X$ if~$y$ has both a neighbour and a non-neighbour in~$X$. 
The set~$X$ is a {\em module} of~$G$ if no vertex in $V\setminus X$ distinguishes~$X$.
A module~$X$ is {\em non-trivial} if $1<|X|<|V|$, otherwise it is {\em trivial}. 
A graph is {\em prime} if it has only trivial modules.

A {\em quasi order}~$\leq$ on a set~$X$ is a reflexive, transitive binary relation.
Two elements $x,y \in X$ in this quasi-order are {\em comparable} if $x \leq y$ or $y \leq x$, otherwise they are {\em incomparable}.
A set of elements in a quasi-order is a {\em chain} if every pair of elements is comparable and it is an {\em antichain} if every pair of elements is incomparable.
The quasi-order~$\leq$ is a {\em well-quasi-order}
if any infinite sequence of elements $x_1,x_2,x_3,\ldots$ in~$X$
contains a pair $(x_i,x_j)$ with $x_i \leq x_j$ and $i<j$.
Equivalently, a quasi-order is a well-quasi-order
if and only if it has no infinite strictly decreasing sequence $x_1 \gneq x_2 \gneq x_3 \gneq \cdots$ and no infinite antichain.

For an arbitrary set~$M$, let~$M^*$ denote the set of finite sequences of elements of~$M$.
Any quasi-order~$\leq$ on~$M$ defines a quasi-order~$\leq^*$ on~$M^*$ as follows:
$(a_1,\ldots,a_m) \leq^* (b_1,\ldots,b_n)$ if and only if there is a sequence of integers $i_1,\ldots,i_m$ with $1 \leq i_1<\cdots<i_m \leq n$ such that $a_j \leq b_{i_j}$ for $j \in \{1,\ldots,m\}$.
We call~$\leq^*$ the {\em subsequence relation}.

\begin{lemma}[Higman's Lemma~\cite{Higman52}]\label{lem:higman}
If $(M,\leq)$ is a well-quasi-order then ${(M^*,\leq^*)}$ is a well-quasi-order.
\end{lemma}

\subsection*{Labelled Induced Subgraphs and Uniformicity}
To define the notion of labelled induced subgraphs, let us consider an arbitrary quasi-order $(W,\leq)$. 
We say that~$G$ is a {\em labelled} graph if each vertex~$v$ of~$G$ is equipped with an element $l_G(v)\in W$ (the {\em label} of~$v$).
Given two labelled graphs~$G$ and~$H$, we say that~$G$ is a {\em labelled induced subgraph} of~$H$ if~$G$ is isomorphic to an induced subgraph of~$H$ and there is an isomorphism that maps each vertex~$v$ of~$G$ to a vertex~$w$ of~$H$ with $l_G(v)\leq l_H(w)$. 
Clearly, if $(W,\leq)$ is a well-quasi-order, then a class of graphs~$X$ cannot contain an infinite sequence of labelled graphs that is strictly-decreasing with respect to the labelled induced subgraph relation.
We therefore say that a class of graphs~$X$ is well-quasi-ordered by the {\em labelled} induced subgraph relation if it contains no infinite antichains of labelled graphs whenever $(W,\leq)$ is a 
{\em well}-quasi-order.
Such a class is readily seen to also be well-quasi-ordered by the induced subgraph relation.

We will use the following three results.

\begin{lemma}[\cite{AL15}]\label{lem:p6-bip-wqo}
The class of $P_6$-free bipartite graphs is well-quasi-ordered by the labelled induced subgraph relation. 
\end{lemma}

\begin{lemma}[\cite{AL15}]\label{lem:Pk-Kl-Kmm-wqo}
Let $k,\ell,m$ be positive integers.
Then the class of $(P_k,K_\ell,\allowbreak K_{m,m})$-free graphs is well-quasi-ordered by the labelled induced subgraph relation.
\end{lemma}

\begin{lemma}[\cite{AL15}]\label{lem:prime}
Let~$X$ be a hereditary class of graphs.
Then~$X$ is well-quasi-ordered by the labelled induced subgraph relation if and only if the set of prime graphs in~$X$ is.
In particular, $X$ is well-quasi-ordered by the labelled induced subgraph relation if and only if the set of connected graphs in~$X$ is. 
\end{lemma}

Let~$k$ be a natural number, let $K$ be a symmetric square $0,1$ matrix of order~$k$, and let~$F_k$ be a graph on the vertex set
$\{1, 2, \ldots , k\}$. Let~$H$ be the disjoint union of infinitely many copies of~$F_k$, and for $i = 1, \ldots , k$, let~$V_i$ be the subset of~$V(H)$ containing vertex~$i$ from each copy of~$F_k$. Now we construct from~$H$ an infinite graph~$H(K)$ on the same vertex set by applying a subgraph complementation to~$V_i$ if and only if $K(i,i)=1$ and by applying bipartite complementation to 
a pair $V_i,V_j$ if and only if $K(i,j)=1$. In other words, two vertices $u\in V_i$ and $v\in V_j$ are adjacent in~$H(K)$ 
if and only if $uv \in E(H)$ and $K(i, j) = 0$ or $uv \notin E(H)$ and $K(i, j) = 1$. Finally,
let ${\cal P}(K, F_k)$ be the hereditary class consisting of all the finite induced subgraphs of~$H(K)$.

Let~$k$ be a natural number.
A graph~$G$ is $k$-{\em uniform} if there is a matrix~$K$ and a graph~$F_k$ such that $G\in {\cal P}(K, F_k)$.
The minimum~$k$ such that~$G$ is $k$-uniform is the {\em uniformicity} of~$G$. 

The following result was proved by Korpelainen and Lozin. 
The class of disjoint unions of cliques is a counterexample for the reverse implication.

\begin{theorem}[\cite{KL11}]\label{thm:uniform}
Any class of graphs of bounded uniformicity is well-quasi-ordered by the labelled induced subgraph relation.
\end{theorem}

\section{Permitted Graph Operations}\label{s-permitted}

It is not difficult to see that if~$G$ is an induced subgraph of~$H$, then~$\overline{G}$ is an induced subgraph of~$\overline{H}$. 
Therefore, a graph class~$X$ is well-quasi-ordered by the induced subgraph relation if and only if the set of complements of graphs in~$X$ is. In this section,
we strengthen this observation in several ways. 

First, we define the operation of subgraph complementation as follows. 

\begin{definition}
{\em Subgraph complementation} in a graph~$G$ is the operation of complementing a subgraph of~$G$ induced by a subset of its vertices. 
\end{definition}

Applied to the entire vertex set of~$G$, this operation coincides with the usual complementation of~$G$.
However, applied to a pair of vertices, it changes the adjacency of these vertices only.
Clearly, repeated applications of this operation can transform~$G$ into any other graph on the same vertex set.
Therefore, unrestricted applications of subgraph complementation may transform a well-quasi-ordered class~$X$ into a class containing infinite antichains.
However, if we bound the number of applications of this operation by a constant, we preserve many nice properties of~$X$, including well-quasi-orderability with respect to the labelled induced subgraph relation.

Next, we introduce the following operation:
\begin{definition}
{\em Bipartite complementation} in a graph~$G$ is the operation of complementing the edges between two disjoint subsets $X,Y \subseteq V(G)$. 
\end{definition}
Note that applying a bipartite complementation between~$X$ and~$Y$  has the same effect as applying a sequence of three subgraph complementations: with respect to~$X$, $Y$ and $X\cup Y$.

Finally, we define the following operation:
\begin{definition}
{\em Vertex deletion} in a graph~$G$ is the operation of removing a single vertex~$v$ from a graph, together with any edges incident to~$v$.
\end{definition}

\subsection{Operations on Labelled Graphs}

Let $k\geq 0$ be a constant and let~$\gamma$ be a graph operation.
A graph class~${\cal G'}$ is {\em $(k,\gamma)$-obtained} from a graph class~${\cal G}$
if the following two conditions hold:
\begin{enumerate}[(i)]
\item every graph in~${\cal G'}$ is obtained from a graph in~${\cal G}$ by performing~$\gamma$ at most~$k$ times, and
\item for every $G\in {\cal G}$ there exists at least one graph 
in~${\cal G'}$ obtained from~$G$ by performing~$\gamma$ at most~$k$ times.
\end{enumerate}

\noindent
We say that~$\gamma$ {\em preserves} well-quasi-orderability by
the labelled induced subgraph relation if for any finite constant~$k$ and any graph class~${\cal G}$, any graph class~${\cal G}'$ that is $(k,\gamma)$-obtained from~${\cal G}$ is well-quasi-ordered by this relation if and only if~${\cal G}$ is.

\newpage
\begin{lemma}\label{lem:lwqo-operations}
The following operations preserve well-quasi-orderability by the labelled induced subgraph relation:
\begin{enumerate}[(i)]
\renewcommand{\theenumi}{\thelemma.(\roman{enumi})}
\renewcommand{\labelenumi}{(\roman{enumi})}
\item\label{lem:subgraph-complementation} Subgraph complementation,
\item\label{lem:bipartite-complementation} Bipartite complementation and
\item\label{lem:adding-vertices} Vertex deletion.
\end{enumerate}
\end{lemma}

\begin{proof}
We start by proving the lemma for subgraph complementations.

Let~$X$ be a class of graphs and let~$Y$ be a set of graphs obtained from~$X$ by applying a subgraph complementation to each graph in~$X$.
More precisely, for each graph $G \in X$ we choose a set~$Z_G$ of vertices in~$G$; we let~$G'$ be the graph obtained from~$G$ by applying a complementation with respect to the subgraph induced by~$Z_G$ and we let~$Y$ be the set of graphs~$G'$ obtained in this way.
Clearly it is sufficient to show that~$X$ is well-quasi-ordered by the labelled induced subgraph relation if and only if~$Y$ is.

Suppose that~$X$ is not well-quasi-ordered under the labelled induced subgraph relation.
Then there must be a well-quasi-order $(L,\leq)$ and an infinite sequence $G_1,G_2,\ldots$ of graphs
in~${\cal X}$
with vertices labelled with elements of~$L$, such that these graphs form an infinite antichain under the labelled induced subgraph relation.
Let $(L',\leq')$ be the quasi-order with $L' = \{(k,l): k \in \{0,1\},l \in L\}$
and $(k,l) \leq' (k',l')$ if and only if $k=k'$ and $l \leq l'$ (so~$L'$ is the disjoint union of two copies of~$L$, where elements of one copy are incomparable with elements of the other copy).
Note that $(L',\leq')$ is a well-quasi-order since $(L,\leq)$ is a well-quasi-order.

For each graph~$G_i$ in this sequence, with labelling~$l_i$, we construct the graph~$G_i'$ (recall that~$G_i'$ is obtained from~$G_i$ by applying a complementation on the vertex set~$Z_{G_i}$).
We label the vertices of~$V(G_i')$ with a labelling~$l_i'$ as follows:
\begin{itemize}
\item set $l_i'(v)=(1,l_i(v))$ if $v \in Z_{G_i}$ and
\item set $l_i'(v)=(0,l_i(v))$ otherwise.
\end{itemize}

We claim that when $G_1',G_2',\ldots$ are labelled in this way they form an infinite antichain with respect to the labelled induced subgraph relation.
Indeed, suppose for contradiction that~$G_i'$ is a labelled induced subgraph of~$G_j'$ for some $i\neq j$.
This means that there is a injective map $f : V(G_i') \to V(G_j')$ such that $l_i'(v) \leq' l_j'(f(v))$ for all $v \in V(G_i')$ and $v,w \in V(G_i')$ are adjacent in~$G_i'$ if and only if~$f(v)$ and~$f(w)$ are adjacent in~$G_j'$.
Now since $l_i'(v) \leq' l_j'(f(v))$ for all $v \in V(G_i')$, by the definition of~$\leq'$ we conclude the following:
\begin{itemize}
\item $l_i(v) \leq l_j(f(v))$ for all $v \in V(G_i')$ and
\item $v \in Z_{G_i}$ if and only if $f(v) \in Z_{G_j}$.
\end{itemize}

Suppose $v,w \in V(G_i)$ with $w \notin Z_{G_i}$ ($v$ may or may not belong to $Z_{G_i}$) and note that this implies $f(w) \notin Z_{G_j}$.
Then~$v$ and~$w$ are adjacent in~$G_i$ if and only if~$v$ and~$w$ are adjacent in~$G_i'$ if and only if~$f(v)$ and~$f(w)$ are adjacent in~$G_j'$ if and only if~$f(v)$ and~$f(w)$ are adjacent in~$G_j$.

Next suppose $v,w \in Z_{G_i}$, in which case $f(v),f(w) \in Z_{G_j}$.
Then~$v$ and~$w$ are adjacent in~$G_i$ if and only if~$v$ and~$w$ are non-adjacent in~$G_i'$ if and only if~$f(v)$ and~$f(w)$ are non-adjacent in~$G_j'$ if and only if~$f(v)$ and~$f(w)$ are adjacent in~$G_j$.

It follows that~$f$ is an injective map $f : V(G_i) \to V(G_j)$ such that $l_i(v) \leq l_j(f(v))$ for all $v \in V(G_i)$ and $v,w \in V(G_i)$ are adjacent in~$G_i$ if and only if~$f(v)$ and~$f(w)$ are adjacent in~$G_j$.
In other words~$G_i$ is a labelled induced subgraph of~$G_j$.
This contradiction means that if $G_1,G_2,\ldots$ is an infinite antichain then $G_1',G_2',\ldots$ must also be an infinite antichain.

Therefore, if the class~$X$ is not well-quasi-ordered by the labelled induced subgraph relation then neither is~$Y$.
Repeating the argument with the roles of $G_1,G_2,\ldots$ and $G_1',G_2',\ldots$ reversed shows that if~$Y$ is not well-quasi-ordered under the labelled induced subgraph relation then neither is~$X$.
This completes the proof for subgraph complementations.

\bigskip
\smallskip
\indent
Since a bipartite complementation is equivalent to doing three subgraph complementations one after another, the result for bipartite complementations follows.

\bigskip
\smallskip
Finally, we 
prove the result for vertex deletions.

Let~$X$ be a class of graphs and let~$Y$ be a set of graphs obtained from~$X$ by deleting exactly one vertex~$z_G$ from each graph~$G$ in~$X$.
We denote the obtained graph by $G-z_G$.
Clearly it is sufficient to show that~$X$ is well-quasi-ordered by the labelled induced subgraph relation if and only if~$Y$ is.

Suppose that~$Y$ is well-quasi-ordered by the labelled induced subgraph relation.
We will show that~$X$ is also well-quasi-ordered by this relation.
For each graph $G \in X$, let~$G'$ be the graph obtained from~$G$ by applying a bipartite complementation between~$\{z_G\}$ and $N(z_G)$, so~$z_G$ is an isolated vertex in~$G'$.
Let~$Z$ be the set of graphs obtained in this way.
By Lemma~\ref{lem:bipartite-complementation}, $Z$ is well-quasi-ordered by the labelled induced subgraph relation if and only if~$X$ is.
Suppose $G_1,G_2$ are graphs in~$Z$ with vertices labelled from some well-quasi-order $(L,\leq)$.
Then for $i \in \{1,2\}$ the vertex~$z_{G_i}$ has a label from~$L$ and the graph $G_i - z_{G_i}$ belongs to~$Y$.
Furthermore if $G_1 - z_{G_1}$ is a labelled induced subgraph of $G_2 - z_{G_2}$ and $l_{G_1}(z_{G_1}) \leq l_{G_2}(z_{G_2})$ then~$G_1$ is a labelled induced subgraph of~$G_2$.
Now by Lemma~\ref{lem:higman} it follows that~$Z$ is well-quasi-ordered by the labelled induced subgraph relation.
Therefore~$X$ is also well-quasi-ordered by this relation.

Now suppose that~$Y$ is not well-quasi-ordered by the labelled induced subgraph relation.
Then~$Y$ contains an infinite antichain $G_1,G_2,\ldots$ with the vertices of~$G_i$ labelled by functions~$l_i$ which takes values in some well-quasi-order $(L,\leq)$.
For each~$G_i$, let~$G_i'$ be a corresponding graph in~$X$, so $G_i=G_i' - z_{G_i'}$.
Then in~$G_i'$ we label~$z_{G_i'}$ with a new label~$*$ and label all other vertices $v \in V(G_i')$ with the same label as that used in~$G_i$.
We make this new label~$*$ incomparable to all the other labels in~$L$ and note that the obtained quasi order $(L \cup \{*\},\leq)$ is also a well-quasi-order.
It follows that $G_1',G_2',\ldots$ is an antichain in~$X$ when labelled in this way.
Therefore, if~$Y$ is not well-quasi-ordered by the labelled induced subgraph relation then neither is~$X$.
This completes the proof.\qed
\end{proof}

Note that the above lemmas only apply to well-quasi-ordering with respect to the {\em labelled} induced subgraph relation.
Indeed, if we take a cycle and delete a vertex, complement the subgraph induced by an edge or apply a bipartite complementation to two adjacent vertices, we obtain  a path.
However, while the set of cycles is an infinite antichain with respect to the induced subgraph relation, the set of paths is not.

\subsection{Operations on $k$-Uniform Graphs}

We now show that our graph operations do not change uniformicity by ``too much'' either. The result for vertex deletion was proved by Korpelainen and Lozin.

\begin{lemma}
Let~$G$ be a graph of uniformicity~$k$.
Let $G'$, $G''$ and $G'''$ be graphs obtained from~$G$ by applying one vertex deletion, subgraph complementation or bipartite complementation, respectively.
Let $\ell'$, $\ell''$ and~$\ell'''$ be the uniformicities of $G'$, $G''$ and~$G'''$, respectively.
Then the following three statements hold:
\begin{enumerate}[(i)]
\renewcommand{\theenumi}{\thelemma.(\roman{enumi})}
\renewcommand{\labelenumi}{(\roman{enumi})}
\item \label{lem:uniform-k-del-vert}$\ell' < k < 2\ell'+1$~\emph{\cite{KL11}};
\item \label{lem:k-uniform-comp}$\frac{k}{2} \leq \ell'' \leq 2k$;
\item \label{lem:k-uniform-bip}$\frac{k}{3} \leq \ell''' \leq 3k$.
\end{enumerate}
\end{lemma}

\begin{proof}
The first inequality of Part~(i) is trivial. The second inequality of Part~(i) was proved in~\cite{KL11}.
We prove Parts~(ii) and~(iii) below.

\medskip
\noindent
Let~$G$ be a graph of uniformicity~$k$, let~$X$ be a set of vertices in~$G$ and let~$G''$ be the graph obtained from~$G$ by applying a complementation with respect to the subgraph induced by~$X$.
Let~$\ell''$ be the uniformicity of~$G''$.
By symmetry, to prove Part~(ii), it is sufficient to prove that $\ell'' \leq 2k$.

Since~$G$ is a $k$-uniform graph, it must belong to $P(F_k,K)$ for some~$F_k$ and some~$K$, so it is an induced subgraph of~$H(K)$.

Consider the graph obtained from~$F_k$ by replacing each vertex~$v$ of~$F_k$ by two non-adjacent vertices~$v$ and~$v'$.
Apply a complementation with respect to $\{1',2',\ldots,k'\}$ and let~$F_k'$ be the obtained graph.
In other words, if $v \in V(F_k)$ then
\begin{align*}
N_{F_k'}(v) &= N_{F_k}(v) \cup \{w' \;|\; w \in N_{F_k}(v)\}\; \textrm{and}\\
N_{F_k'}(v') &= N_{F_k}(v) \cup \{w' \;|\; w \in V(F_k) \setminus N_{F_k}[v]\}. 
\end{align*}

Let~$K'$ be a $2k \times 2k$ matrix indexed by $\{1,2,\ldots,k,1',2',\ldots,k'\}$ with $K_{i,j}'=K_{i,j'}'=K_{i',j}'=1-K_{i',j'}'=K_{i,j}$.

This means that~$H(K')$ is formed from~$H(K)$ by adding a copy of each vertex that has the same neighbourhood and then applying a complementation with respect to the set of newly-created vertices.

Similarly, $G''$ can be obtained from~$G$ by replacing each vertex in~$X$ by a copy with the same neighbourhood and then applying a complementation with respect to the set of newly-created vertices.
Therefore~$G''$ is an induced subgraph of~$H(K')$.
Therefore~$G''$ is a $2k$-uniform graph.

\medskip
\noindent
Part~(iii) follows from similar arguments to those for Part~(ii).
(Also note that if we weaken the bounds in Part~(iii) to $\frac{k}{8} \leq \ell''' \leq 8k$ then the result follows immediately from combining Part~(ii) with the fact that a bipartite complementation is equivalent to a sequence of three subgraph complementations.)\qed
\end{proof}

\section{A New Well-Quasi-Ordered Class}\label{s-yes}

In this section we show that $(\overline{2P_1+P_2},P_2+\nobreak P_3)$-free graphs are well-quasi-ordered by the labelled induced subgraph relation.
We divide the proof into several sections, depending on whether or not the graphs under consideration contain certain induced subgraphs or not.
We follow the general scheme that Dabrowski, Huang and Paulusma~\cite{DHP0} used to prove that this class has bounded clique-width, but we will also need a number of new arguments.

\subsection{Graphs containing a~$K_5$}

We first consider graphs containing a~$K_5$ and prove the following lemma.

\begin{lemma}\label{lem:diamond-P_2+P_3-free-K5-non-free}
The class of $(\overline{2P_1+P_2},P_2+\nobreak P_3)$-free graphs that contain a~$K_5$ is well-quasi-ordered by the labelled induced subgraph relation.
\end{lemma}

\begin{proof}
Let~$G$ be a $(\overline{2P_1+P_2},P_2+P_3)$-free graph.
Let~$X$ be a maximal (by set inclusion) clique in~$G$ containing at least five
vertices. 

\medskip
\clm{\em\label{clm:at-most-one-nbr-in-X}Every vertex not in~$X$ has at most one neighbour in~$X$.}
This follows from the fact that~$X$ is maximal and~$G$ is $\overline{2P_1+P_2}$-free.

\medskip
\noindent
Suppose there is a~$P_3$ in $G \setminus X$, say on vertices $x_1,x_2,x_3$ in
that order. Since $|X| \geq 5$, we can find $y_1,y_2 \in X$ none of which are
adjacent to any of $x_1,x_2,x_3$. Then $G[y_1,y_2,x_1,x_2,x_3]$ is a $P_2+\nobreak P_3$.
This contradiction shows that $G \setminus X$ is $P_3$-free and must therefore be a union of disjoint
cliques $X_1,\ldots,X_k$.
We say that a clique~$X_i$ is {\em large} if it contains at least two vertices and that it is {\em small} if it contains exactly one vertex.

\medskip
\clm{\em\label{clm:nbrs-in-large-Xj}If $x \in X$ is adjacent to $y \in X_i$ and~$X_j$ (with $i \neq j$) is large, then~$x$ can have at most one non-neighbour in~$X_j$.}
For contradiction, assume that~$x$ is non-adjacent to $z_1,z_2 \in
X_j$. Since $|X| \geq 5$ and each vertex that is not in~$X$
has at most one neighbour in~$X$, there must be a vertex $x' \in X$ that is
non-adjacent to $y,z_1$ and~$z_2$. Then $G[z_1,z_2,x',x,y]$ is a $P_2+P_3$, a contradiction.

\medskip
\noindent
We consider several cases:

\thmcase{\label{case:1-clique}$G \setminus X$ contains at most one clique.}
Then the complement of~$G$ is a bipartite graph. Moreover, since the complement of $\overline{2P_1+P_2}$
is an induced subgraph of~$P_6$, we conclude that the complements of graphs in our class form a subclass of $P_6$-free bipartite graphs, which are well-quasi-ordered by the labelled 
induced subgraph relation by Lemmas~\ref{lem:p6-bip-wqo} and~\ref{lem:subgraph-complementation}.

\thmcase{\label{case:0-large}$G \setminus X$ does not contain large cliques.}
In this case, the structure of graphs can be described as follows: take a
collection of stars and create a clique on their central vertices and then add a number (possibly zero) of isolated vertices.
In other words, applying subgraph complementation once (to the clique~$X$), we obtain a graph which is  a disjoint union of stars and isolated vertices.
Clearly, a graph every connected component of which is a star or an isolated vertex is $P_6$-free bipartite and hence by Lemmas~\ref{lem:p6-bip-wqo} and~\ref{lem:subgraph-complementation}
we conclude that graphs in our class are well-quasi-ordered by the labelled induced subgraph relation. 

\thmcase{\label{case:1-large}$G \setminus X$ contains exactly one large clique.}
Without loss of generality, assume that~$X_1$ is large and the remaining cliques $X_2,\ldots,X_k$ are small.
Suppose there are~$\ell$ distinct vertices $x_1,\ldots,x_\ell \in X$, each of which has a neighbour in $X_2 \cup \cdots \cup X_k$.
By Claim~\ref{clm:nbrs-in-large-Xj}, each of these vertices has at most one non-neighbour in~$X_1$.
But then $\ell\le 2$, since otherwise a vertex of~$X_1$ has more than one neighbour in~$X$. 

Therefore, by deleting at most two vertices from~$G$ we transform it to a graph from Case~\ref{case:1-clique} plus a number of isolated vertices. 
Lemma~\ref{lem:prime} allows us to ignore isolated vertices, while Lemma~\ref{lem:adding-vertices} allows the deletion of finitely many vertices.
Therefore, in Case~\ref{case:1-large} we deal with a set of graphs which is well-quasi-ordered by the labelled induced subgraph relation. 

\thmcase{\label{case:at-least-2-large}$G \setminus X$ contains at least two cliques that are large.}
Suppose there is a vertex $x \in X$ that has a neighbour outside of~$X$. 
By Claim~\ref{clm:nbrs-in-large-Xj}, $x$ has at most one non-neighbour in each large clique. 
Therefore, at most two vertices of~$X$ have neighbours outside of~$X$, since otherwise each large clique would have a vertex with more than one neighbour in~$X$.
But then by deleting at most two vertices we transform~$G$ into a $P_3$-free graph (i.e. a graph every connected component of which is a clique).
It is well-known (and also follows from Lemma~\ref{lem:prime}) that the set of $P_3$-free graphs is well-quasi-ordered by the labelled induced subgraph relation.
Therefore, by Lemma~\ref{lem:adding-vertices}, the same is true for graphs in Case~\ref{case:at-least-2-large}.
\qed
\end{proof}

\subsection{Graphs containing a~$C_5$}
\noindent
By Lemma~\ref{lem:diamond-P_2+P_3-free-K5-non-free}, we may restrict ourselves to looking at $K_5$-free graphs in our class.
We now consider the case where these graphs have an induced $C_5$.

\begin{lemma}\label{lem:diamond-P_2+P_3-free-C5-non-free}
The class of $(\overline{2P_1+P_2},P_2+\nobreak P_3,K_5)$-free graphs that contain an induced~$C_5$ has bounded uniformicity.
\end{lemma}

\begin{proof}
To prove this, we modify the proof from~\cite{DHP0}, which shows that this class of graphs has bounded clique-width.
Let~$G$ be a $(\overline{2P_1+P_2},P_2+P_3,K_5)$-free graph containing a~$C_5$, say on vertices $v_1,v_2,v_3,v_4,v_5$ in order.
Our goal is to show that the graph~$G$ has bounded uniformicity and hence, by Lemma~\ref{lem:uniform-k-del-vert}, in the proof we can neglect any set of vertices that is bounded in size by a constant.

Let~$Y$ be the set of vertices adjacent to~$v_1$ and~$v_2$ (and possibly to other vertices on the cycle).
If $y_1,y_2 \in Y$ are non-adjacent, then $G[v_1,v_2,y_1,y_2]$ would be a $\overline{2P_1+P_2}$. Therefore, $Y$ is a clique.
This clique has at most two vertices, since otherwise three vertices of~$Y$ together with~$v_1$ and~$v_2$ would create a~$K_5$. 
Therefore, the set of all vertices with two consecutive neighbours on the cycle is finite and hence can be neglected (removed from the graph).
We may therefore assume that each vertex not on the cycle has at most two neighbours on the cycle and if it has two such neighbours, they must be non-consecutive vertices of the cycle.

Now let~$W$ be the set of vertices whose unique neighbour on the cycle is~$v_1$.
If $y_1,y_2 \in W$ are non-adjacent, then $G[v_3,v_4,y_1,v_1,y_2]$ would be a $P_2+\nobreak P_3$.
If $y_1,y_2 \in W$ are adjacent, then $G[y_1,y_2,v_2,v_3,v_4]$ would be a $P_2+\nobreak P_3$.
Therefore,~$W$ contains at most one vertex, and hence the set of all vertices with exactly one neighbour 
on the cycle can be removed from the graph.

Let~$X$ be the set of vertices with no neighbours on the cycle. 
$X$ must be an independent set, since if two vertices in $x_1,x_2 \in X$
are adjacent, then $G[x_1,x_2,v_1,v_2,v_3]$ would induce a $P_2+P_3$ in~$G$.

For $i \in \{1,2,3,4,5\}$, let~$V_i$ be the
set of vertices not on the cycle that are adjacent to~$v_{i-1}$ and~$v_{i+1}$,
but non-adjacent to all other vertices of the cycle (subscripts on vertices and vertex sets are interpreted
modulo~$5$ throughout this proof). For each~$i$, the set~$V_i$ must be independent, since if $x,y \in V_i$ are adjacent
then $G[x,y,v_{i-1},v_{i+1}]$ is a $\overline{2P_1+P_2}$.

We say that two sets~$V_i$ and~$V_j$
are \emph{consecutive} (respectively \emph{opposite}) if~$v_i$ and~$v_j$ are
distinct adjacent (respectively non-adjacent) vertices of the cycle.
We say that a set~$X$ or~$V_i$ is \emph{large} if it contains at least three
vertices, otherwise it is \emph{small}.
We say that a bipartite graph with bipartition classes~$A$ and~$B$ is a {\em matching} ({\em co-matching})
if every vertex in~$A$ has at most one neighbour (non-neighbour) in~$B$, and vice versa.

Dabrowski, Huang and Paulusma proved the following claims about the edges between these sets (see Appendix~\ref{app:c5} for a proof).

\medskip
\clm{\em(\cite{DHP0}) $G[V_i \cup X]$ is a matching.}
\clm{\em(\cite{DHP0}) If~$V_i$ and~$V_j$ are opposite, then $G[V_i\cup V_j]$ is a matching.}
\clm{\em(\cite{DHP0}) If~$V_i$ and~$V_j$ are consecutive, then $G[V_i \cup V_j]$ is a co-matching.}
\clm{\em(\cite{DHP0}) If~$V_i$ is large, then~$X$ is anti-complete to $V_{i-2} \cup V_{i+2}$.}
\clm{\em(\cite{DHP0}) If~$V_i$ is large, then~$V_{i-1}$ is anti-complete to~$V_{i+1}$.}
\clm{\em(\cite{DHP0}) If $V_{i-1},V_{i},V_{i+1}$ are large, then~$V_i$ is complete to
$V_{i-1} \cup V_{i+1}$.}

We also prove the following claim:

\medskip
\clm{\em\label{claim:last}Suppose two consecutive sets~$V_i$ and~$V_{i+1}$ are large and a vertex $y\in V_i$ is not adjacent to a vertex $z\in V_{i+1}$.
Then every vertex $x\in X\cup V_{i+3}$ is either complete or anti-complete to $\{y,z\}$.}
To prove this, suppose for contradiction that this is not the case.
Without loss of generality, we may assume that that~$x$ is adjacent to~$y$ but not to~$z$.
Since~$V_{i}$ is large it contains at least two other vertices, say~$a$ and~$b$.
Then~$z$ is adjacent to both~$a$ and~$b$ (since $G[V_i\cup V_{i+1}]$ is a co-matching), while~$x$ is adjacent neither to~$a$ nor to~$b$ (since $G[V_i \cup X]$ and  $G[V_i\cup V_{i+3}]$ are matchings).
But then $G[x,y,a,z,b]$ is a $P_2+P_3$, a contradiction.

\medskip
\noindent
We are now ready to prove the lemma.
We may delete from~$G$ the vertices $v_1,\ldots, v_5$ and all vertices in every small set~$X$ or~$V_i$. Let~$G'$ be the resulting graph.  
In order to show that~$G'$ has bounded uniformicity, we split the analysis into the following cases.

\medskip
\thmcase{All sets $V_1,\ldots,V_5$ are large.}
From the above claims we conclude that any two consecutive sets are complete to each other and any two opposite sets are anti-complete to each other. 
Also, $X$ is anti-complete to each of them. Therefore $G'$ is 6-uniform.

\medskip
\thmcase{Four sets~$V_i$ are large, say $V_1,\ldots,V_4$.} Then~$V_1$ and~$V_4$ form a matching, while any other pair of these sets are either complete or anti-complete to each other.
Also, $X$ is anti-complete to each of them. Therefore $G'$ is 5-uniform. 

\medskip
\thmcase{Three consecutive sets~$V_i$ are large, say $V_1,V_2,V_3$.} Then~$V_2$ is complete to~$V_1$ and~$V_3$, while~$V_1$ and~$V_3$ are anti-complete to each-other. Also, $X$ is anti-complete to~$V_1$ and~$V_3$ forms a matching with~$V_2$.
Therefore $G'$ is 4-uniform.

\medskip
\thmcase{Three non-consecutive sets~$V_i$ are large, say $V_1,V_3,V_4$.} Then~$X$ is anti-complete to each of them. From the above claims we know that~$V_1$ forms a matching with both~$V_3$ and~$V_4$,
while~$V_3$ and~$V_4$ form a co-matching. Also, from Claim~\ref{claim:last} we conclude that whenever two vertices $y\in V_3$ and $z\in V_4$ are non-adjacent, then either none of them 
has a neighbour in~$V_1$ or they both are adjacent to the same vertex of~$V_1$. Therefore, if we complement the edges between~$V_3$ and~$V_4$, then~$G'$ transforms into a graph 
in which every connected component is one of the following graphs: $K_3$, $P_3$, $P_2$,~$P_1$.
Each of these graphs is an induced subgraph of $\overline{P_1+P_3}$ (also known as the {\em paw}), so the obtained graph is $4$-uniform.
By Lemma~\ref{lem:k-uniform-bip}, it follows that $G[V_1\cup V_3\cup V_4]$ is $12$-uniform and so~$G'$ is $13$-uniform.

\medskip
\thmcase{Two consecutive sets~$V_i$ are large, say $V_3,V_4$.} This case is similar to the previous one, where the role of~$V_1$ is played by~$X$.
Thus $G'$ is $12$-uniform. 

\medskip
\thmcase{Two non-consecutive sets~$V_i$ are large.} Then~$X$ is anti-complete to each of them and hence the graph is obviously $3$-uniform. 

\medskip
\thmcase{At most one set~$V_i$ is large.} Then~$G'$ is obviously $2$-uniform.

\medskip
\noindent
Since the above cases cover all possibilities, this completes the proof.\qed
\end{proof}

\subsection{Graphs containing a~$C_4$}

By Lemmas~\ref{lem:diamond-P_2+P_3-free-K5-non-free} and~\ref{lem:diamond-P_2+P_3-free-C5-non-free}, we may restrict ourselves to looking at $(K_5,C_5)$-free graphs in our class.
We prove the following structural result.

\begin{lemma}\label{lem:diamond-P_2+P_3-free-C4-non-free}
Let~$G$ be a $(\overline{2P_1+P_2},P_2+\nobreak P_3,K_5,C_5)$-free graph containing an induced~$C_4$.
Then by deleting at most $17$ vertices and applying at most two bipartite complementations, we can modify~$G$ into the disjoint union of a $(P_2+P_3)$-free bipartite graph and a $3$-uniform graph.
\end{lemma}

\begin{proof}
In order to prove the lemma, we again modify the proof from~\cite{DHP0}, which shows that this class of graphs has bounded clique-width.
Let~$G$ be a $(\overline{2P_1+P_2},\allowbreak P_2+\nobreak P_3,\allowbreak K_5,C_5)$-free graph containing a~$C_4$ induced by the vertices $v_1,v_2,v_3,v_4$ in order.
We interpret subscripts on vertices modulo~$4$ in this proof.

Let~$Y$ be the set of vertices adjacent to~$v_1$ and~$v_2$ (and possibly to other vertices on the cycle).
If $y_1,y_2 \in Y$ are non-adjacent, then $G[v_1,v_2,y_1,y_2]$ would be a $\overline{2P_1+P_2}$. Therefore, $Y$ is a clique.
This clique has at most two vertices, since otherwise three vertices of~$Y$ together with~$v_1$ and~$v_2$ would create a~$K_5$. 
Therefore, after deleting at most $2 \times 4 = 8$ vertices, we may assume that no vertex of the graph contains two consecutive neighbours on the cycle.

Let~$V_1$ denote the set of vertices adjacent to~$v_2$ and~$v_4$, and let~$V_2$ denote the set of vertices adjacent to~$v_1$ and~$v_3$.
For $i\in \{1,2,3,4\}$, let~$W_i$ denote the set of vertices whose only neighbour on the cycle is~$v_i$.
If a set~$W_i$ contains at most one vertex then we may delete this vertex.
Thus, deleting at most four vertices, we may assume that every set~$W_i$ is either empty of contains at least two vertices.
Finally, let~$X$ be the set of vertices with no neighbour on the cycle.

Dabrowski, Huang and Paulusma proved the following claims about the edges between these sets (see Appendix~\ref{app:c4-1} for a proof).

\medskip
\clm{\em(\cite{DHP0}) $V_i$ is independent for $i\in \{1,2\}$.}
\clm{\em(\cite{DHP0}) $W_i$ is independent for $i\in \{1,2,3,4\}$.}
\clm{\em(\cite{DHP0}) $X$ is independent.}
\clm{\em(\cite{DHP0}) $W_i$ is anti-complete to~$X$ for $i\in \{1,2,3,4\}$.}
\clm{\em(\cite{DHP0}) For $i \in \{1,2\}$ either~$W_i$ or~$W_{i+2}$ is empty. Therefore, we may assume by symmetry that $W_3=\emptyset$ and $W_4=\emptyset$.}

Note that in our arguments so far we have deleted at most $12$ vertices. We now argue as follows:

Any vertices of~$X$ that do not have neighbours in $V_1\cup V_2$ must be isolated vertices of the graph.
Since adding isolated vertices to a $P_2+\nobreak P_3$-free bipartite graph maintains the property of it being $P_2+\nobreak P_3$-free and bipartite, we may therefore assume that every vertex in~$X$ has a neighbour in $V_1\cup V_2$.
Let~$X_0$ denote the subset of~$X$ whose vertices have neighbours both in~$V_1$ and~$V_2$, 
let~$X_1$ denote the subset of~$X$ whose vertices have no neighbours in~$V_1$ and let~$X_2$ denote the subset of~$X$ whose vertices have no neighbours in~$V_2$.

Let~$V_0$ denote the set of vertices in $V_1\cup V_2$ adjacent to at least one vertex of~$X_0$ and let $V_{10}=V_1\cap V_0$ and $V_{20}=V_2\cap V_0$.
If~$V_{10}$ or~$V_{20}$ contains at most one vertex then we may delete this vertex.
This would cause~$X_0$ to become empty.
Therefore, by deleting at most one vertex, we may assume that either both~$V_{10}$ and~$V_{20}$ each contain at least two vertices or else $V_{10}$, $V_{20}$ and~$X_0$ are all empty.
We will show that $G[X_0 \cup V_0]$ is $3$-uniform and can be separated from the rest of the graph using at most two bipartite complementations.
To do this, we first prove the following additional claims.

\medskip
\clm{\em\label{clm:x0-unique-vi0-nbr}Every vertex of~$X_0$ has exactly one neighbour in~$V_{10}$ and exactly one neighbour in~$V_{20}$ and these neighbours are adjacent.} 
First, we observe that if a vertex $x\in X_0$ is adjacent to $y\in V_1$ and to $z\in V_2$,
then~$y$ is adjacent to~$z$, since otherwise $G[x,y,v_2,v_1,z]$ is an induced~$C_5$. This implies that if~$x$ has the third neighbour $y'\in V_1\cup V_2$, then $G[x,z,y,y']$ is a $\overline{2P_1+P_2}$.
This contradiction proves the claim.

\medskip
\clm{\em\label{clm:v0-unique-x0-nbr}Every vertex of~$V_0$ is adjacent to exactly one vertex of~$X_0$.}
Let~$v$ be a vertex in~$V_0$.
Without loss of generality, assume that~$v$ belongs to~$V_{10}$.
Suppose that~$v$ has at least two neighbours in~$X_0$, say~$x$ and~$x'$, and at least one non-neighbour, say~$x''$.
Let~$v''$ be the neighbour of~$x''$ in~$V_{10}$.
Then $G[v'',x'',x,v,x']$ is a $P_2+\nobreak P_3$. This contradiction
shows that if~$v$ has at least two neighbours in~$X_0$, then it must be adjacent to all the vertices of~$X_0$.
Since every vertex of~$X_0$ has exactly one neighbour in~$V_{10}$ it follows that~$v$ is the only vertex of~$V_{10}$, a contradiction.
We conclude that~$v$ (and hence every other vertex of~$V_0$) has exactly one neighbour in~$X_0$.

\medskip
\clm{\em\label{clm:v10-comp-v2}$V_{10}$ is complete to~$V_2$ and~$V_{20}$ is complete to~$V_1$.}
Suppose $v\in V_{10}$ is non-adjacent to $y\in V_2$ and let~$x$ be the unique neighbour of~$v$ in~$X_{0}$.
Since~$y$ is non-adjacent to~$v$, it cannot be the unique neighbour of~$x$ in~$V_2$.
Therefore~$y$ must be non-adjacent to~$x$.
It follows that $G[x,v,v_1,y,v_3]$ is a $P_2+\nobreak P_3$, a contradiction.
The second part of the claim follows by symmetry.

\medskip
\clm{\em\label{clm:w1w2x1x2-comp-vi0}Every vertex in $W_1\cup W_2\cup X_1\cup X_2$ is either complete or anti-complete to $V_{i0}$ for $i=1,2$.} Suppose a vertex $w\in W_1\cup W_2\cup X_1\cup X_2$ has both a neighbour~$v$ and a non-neighbour~$v'$ in~$V_{10}$.
Let~$x$ and~$x'$ be the neighbours of~$v$ and~$v'$, respectively, in~$X_0$.
Recall that~$x$ and~$x'$ must be non-adjacent to~$w$.
Then $G[v',x',w,v,x]$ is a $P_2+\nobreak P_3$, a contradiction.

\medskip
\noindent
By Claims~\ref{clm:v10-comp-v2} and~\ref{clm:w1w2x1x2-comp-vi0}, every vertex outside $V_{10}\cup V_{20}\cup X_0$ is either complete or anti-complete to~$V_{10}$ and either complete or anti-complete to~$V_{20}$.
Applying at most two bipartite complementations, we may therefore disconnect $G[V_{10}\cup V_{20}\cup X_0]$ from the rest of the graph i.e. remove all edges between vertices in $V_{10}\cup V_{20}\cup X_0$ and vertices outside $V_{10}\cup V_{20}\cup X_0$.
By Claims~\ref{clm:x0-unique-vi0-nbr}, \ref{clm:v0-unique-x0-nbr} and~\ref{clm:v10-comp-v2} it follows that $G[V_{10}\cup V_{20}\cup X_0]$ is a $3$-uniform graph.

\medskip
\noindent
We may now assume that~$X_0$ is empty.
Let~$H$ be the graph obtained from~$G$ by deleting the vertices of the original cycle.
Note that $V(H)=X_1\cup V_1\cup W_1 \cup X_2\cup V_2\cup W_2$.
It remains to show that~$H$ is a $(P_2+\nobreak P_3)$-free bipartite graph.

We claim that~$H$ is bipartite with independent sets $X_1\cup V_1\cup W_1$ and  $X_2\cup V_2\cup W_2$.
To show this, it suffices to prove that~$H$ has no triangles, because all other odd cycles are forbidden in~$G$ (and hence in~$H$).
We know that $X_1\cup V_1$ is an independent set and $X_2\cup V_2$ is an independent set.
Also, $W_1$ and~$W_2$ are independent and the vertices of~$X$ have no neighbours in $W_1\cup W_2$. 
It follows that in~$H$ vertices in~$X_1$ an can only have neighbours in~$V_2$ and vertices of~$X_2$ can only have neighbours in~$V_1$, so no triangle in~$H$ contains a vertex of~$X$.
By symmetry, if $H[x,y,z]$ is a triangle then we may therefore assume that
$x \in W_1$, $y \in V_2$ and $z \in V_1 \cup W_2$.
Now $G[x,y,z,v_1]$ is a $\overline{2P_1+P_2}$, a contradiction.
It follows that~$H$ is bipartite.
Moreover, it is $(P_2+\nobreak P_3)$-free.
This completes the proof.\qed
\end{proof}

Since $P_2+\nobreak P_3$ is an induced subgraph of~$P_6$, it follows that every $(P_2+\nobreak P_3)$-free graph is $P_6$-free.
Combining Lemma~\ref{lem:diamond-P_2+P_3-free-C4-non-free} with Theorem~\ref{thm:uniform} and Lemmas \ref{lem:p6-bip-wqo}, \ref{lem:bipartite-complementation} and~\ref{lem:adding-vertices} we therefore obtain the following corollary.

\begin{corollary}\label{cor:diamond-P_2+P_3-free-C4-non-free}
The class of $(\overline{2P_1+P_2},P_2+\nobreak P_3,K_5,C_5)$-free graphs containing an induced~$C_4$ is well-quasi-ordered by the labelled induced subgraph relation.
\end{corollary}

\subsection{$(\overline{2P_1+P_2},P_2+\nobreak P_3)$-free graphs}

\begin{theorem}\label{thm:diamond-p2+p3-wqo}
The class of $(\overline{2P_1+P_2},P_2+\nobreak P_3)$-free graphs is well-quasi-ordered by the labelled induced subgraph relation.
\end{theorem}

\begin{proof}
Graphs in the class under consideration containing an induced subgraph isomorphic to
$K_5$,~$C_5$ or~$C_4$ are well-quasi-ordered by the labelled induced subgraph relation by Lemmas~\ref{lem:diamond-P_2+P_3-free-K5-non-free} and~\ref{lem:diamond-P_2+P_3-free-C5-non-free} and Corollary~\ref{cor:diamond-P_2+P_3-free-C4-non-free}, respectively.
The remaining graphs form a subclass of $(P_6,K_{5},K_{2,2})$-free graphs, since $C_4=K_{2,2}$ and $P_2+\nobreak P_3$ is an induced subgraph of~$P_6$.
By Lemma~\ref{lem:Pk-Kl-Kmm-wqo}, this class of graphs is well-quasi-ordered
by the labelled induced subgraph relation. Therefore, the class of $(\overline{2P_1+P_2},P_2+\nobreak P_3)$-free graphs is well-quasi-ordered by the labelled induced subgraph relation.\qed
\end{proof}

\section{Two New Non-Well-Quasi-Ordered Classes}\label{s-no}

In this section we show that the classes of $(\overline{2P_1+P_2},P_2+\nobreak P_4)$-free graphs and $(\overline{P_1+P_4},P_1+\nobreak 2P_2)$-free graphs are not well-quasi-ordered by the induced subgraph relation.
The antichain used to prove the first of these cases was previously used by Atminas and Lozin to show that the class of $(\overline{2P_1+P_2},P_6)$-free graphs is not well-quasi-ordered with respect to the induced subgraph relation.
Because of this, we can show show a stronger result for the first case.

\begin{theorem}\label{thm:diamond-P2+P_4}
The class of $(\overline{2P_1+P_2},P_2+\nobreak P_4,P_6)$-free graphs is not well-quasi-ordered 
by the induced subgraph relation.
\end{theorem}

\begin{proof}
Let $n \geq 2$ be an integer and consider a cycle $C_{4n}$, say $x_1-x_2-\cdots-x_{4n}-x_1$.
We partition the vertices of this cycle into three sets as follows as follows:
\begin{align*}
X &=\{x_i\; |\; i \equiv 0\; \textrm{or}\; 2\; \bmod 4\},\\
Y &=\{x_i\; |\; i \equiv 1\; \bmod 4\},\\
Z &=\{x_i\; |\; i \equiv 3\; \bmod 4\}.
\end{align*}
Let~$G_{4n}$ be the graph obtained from~$C_{4n}$ by connecting every vertex of~$Y$ to every vertex of~$Z$ (see also \figurename~\ref{fig:thm:diamond-P2+P_4}).
Atminas and Lozin showed that the resulting graphs are $(\overline{2P_1+P_2},P_6)$-free and form an infinite antichain with respect to the induced subgraph relation~\cite{AL15}.

\begin{figure}
\begin{center}
\begin{tikzpicture}
\node[circle,draw,inner sep=0pt, text width=15pt, align=center] (x0)  at   (0:3) {$x_{12}$};
\node[circle,draw,inner sep=0pt, text width=15pt, align=center] (x1)  at  (30:3) {$x_1$};
\node[circle,draw,inner sep=0pt, text width=15pt, align=center] (x2)  at  (60:3) {$x_2$};
\node[circle,draw,inner sep=0pt, text width=15pt, align=center] (x3)  at  (90:3) {$x_3$};
\node[circle,draw,inner sep=0pt, text width=15pt, align=center] (x4)  at (120:3) {$x_4$};
\node[circle,draw,inner sep=0pt, text width=15pt, align=center] (x5)  at (150:3) {$x_5$};
\node[circle,draw,inner sep=0pt, text width=15pt, align=center] (x6)  at (180:3) {$x_6$};
\node[circle,draw,inner sep=0pt, text width=15pt, align=center] (x7)  at (210:3) {$x_7$};
\node[circle,draw,inner sep=0pt, text width=15pt, align=center] (x8)  at (240:3) {$x_8$};
\node[circle,draw,inner sep=0pt, text width=15pt, align=center] (x9)  at (270:3) {$x_9$};
\node[circle,draw,inner sep=0pt, text width=15pt, align=center] (x10) at (300:3) {$x_{10}$};
\node[circle,draw,inner sep=0pt, text width=15pt, align=center] (x11) at (330:3) {$x_{11}$};

\foreach \from/\to in {x0/x1,x1/x2,x2/x3,x3/x4,x4/x5,x5/x6,x6/x7,x7/x8,x8/x9,x9/x10,x10/x11,x11/x0}
   \draw (\from) -- (\to);

\foreach \from in {x1,x5,x9}
  \foreach \to in {x3,x7,x11}
     \draw (\from) -- (\to);
\end{tikzpicture}
\end{center}
\caption{The Graph~$G_{4n}$ from Theorem~\ref{thm:diamond-P2+P_4} when $n=3$.}
\label{fig:thm:diamond-P2+P_4}
\end{figure}
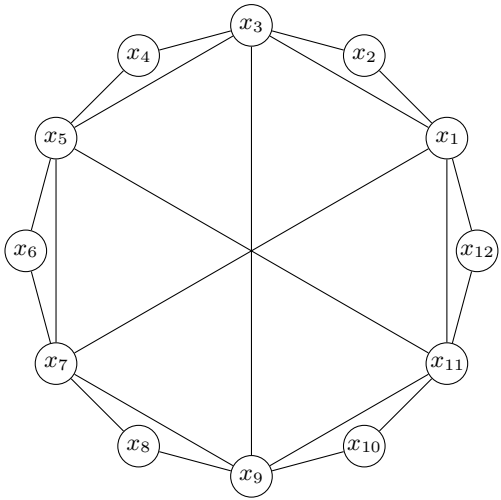

It remains to prove that~$G_{4n}$ is $(P_2+\nobreak P_4)$-free.
We argue as in the proof of~\cite[Theorem~1~(iv)]{DHP0}.
For contradiction,
suppose that~$G_{4n}$ contains an induced subgraph~$I$ isomorphic to $P_2+\nobreak P_4$.
Let~$I_1$ and~$I_2$ be the connected components of~$I$ isomorphic to~$P_2$
and~$P_4$, respectively. Since $G_{4n}[Y\cup Z]$ is complete bipartite, $I_2$ must
contain at least one vertex of~$X$. Since the two neighbours of any vertex
of~$X$ are adjacent, any vertex of~$X$ in~$I_2$ must be an end-vertex of~$I_2$.
Then, as~$Y$ and~$Z$ are independent sets, $I_2$ contains a vertex of both~$Y$
and~$Z$. As~$I_1$ can contain at most one vertex of~$X$ (because~$X$ is an
independent set), $I_1$ contains a vertex~$u\in Y\cup Z$. However, $G_{4n}[Y\cup
Z]$ is complete bipartite and~$I_2$ contains a vertex of both~$Y$ and~$Z$.
Hence, $u$ has a neighbour in~$I_2$, which is not possible. This completes the
proof.\qed
\end{proof}

The graphs~$G_{4n}$ in the above proof are obtained from cycles in the same way that walls were transformed in~\cite{DHP0} to show unboundedness of clique-width for $(\overline{2P_1+P_2},P_2+\nobreak P_4)$-free graphs (The sets $A,B$ and~$C$ in~\cite{DHP0} correspond to the sets $Y,X$ and~$Z$, respectively, in the proof above). In fact the construction in~\cite{DHP0} is also $P_6$-free by the same arguments as in~\cite{AL15}, so we obtain the following:
\begin{remark}
The class of $(\overline{2P_1+P_2},P_2+\nobreak P_4,P_6)$-free graphs has unbounded clique-width.
\end{remark}

\noindent
For our second class, we need a new construction.

\newpage
\begin{theorem}\label{thm:gem-P1+2P_2}
The class of $(\overline{P_1+P_4},P_1+\nobreak 2P_2)$-free graphs is not well-quasi-ordered 
by the induced subgraph relation.
\end{theorem}

\begin{proof}
Let $n \geq 3$ be an integer.
Consider a cycle~$C_{4n}$, say $x_1-x_2-\cdots-x_{4n}-x_1$. We partition the vertices of~$C_{4n}$ as follows:
\begin{align*}
X &=\{x_i\; |\; i \equiv 0\; \textrm{or}\; 1\; \bmod 4\},\\
Y &=\{x_i\; |\; i \equiv 2\; \textrm{or}\; 3\; \bmod 4\}.
\end{align*}
Next we apply a complementation to each of~$X$ and~$Y$, so that in the resulting graph~$X$ and~$Y$ each induce a clique on~$2n$ vertices with a perfect matching removed. Let~$G_{4n}$ be the resulting graph (see also \figurename~\ref{fig:thm:gem-P1+2P_2}).

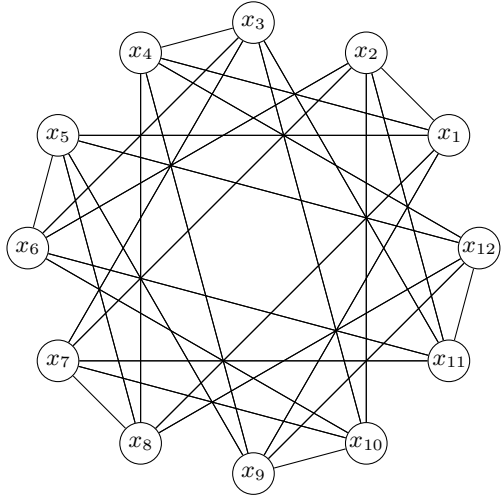
\begin{figure}
\begin{center}
\begin{tikzpicture}
\node[circle,draw,inner sep=0pt, text width=15pt, align=center] (x0)  at   (0:3) {$x_{12}$};
\node[circle,draw,inner sep=0pt, text width=15pt, align=center] (x1)  at  (30:3) {$x_1$};
\node[circle,draw,inner sep=0pt, text width=15pt, align=center] (x2)  at  (60:3) {$x_2$};
\node[circle,draw,inner sep=0pt, text width=15pt, align=center] (x3)  at  (90:3) {$x_3$};
\node[circle,draw,inner sep=0pt, text width=15pt, align=center] (x4)  at (120:3) {$x_4$};
\node[circle,draw,inner sep=0pt, text width=15pt, align=center] (x5)  at (150:3) {$x_5$};
\node[circle,draw,inner sep=0pt, text width=15pt, align=center] (x6)  at (180:3) {$x_6$};
\node[circle,draw,inner sep=0pt, text width=15pt, align=center] (x7)  at (210:3) {$x_7$};
\node[circle,draw,inner sep=0pt, text width=15pt, align=center] (x8)  at (240:3) {$x_8$};
\node[circle,draw,inner sep=0pt, text width=15pt, align=center] (x9)  at (270:3) {$x_9$};
\node[circle,draw,inner sep=0pt, text width=15pt, align=center] (x10) at (300:3) {$x_{10}$};
\node[circle,draw,inner sep=0pt, text width=15pt, align=center] (x11) at (330:3) {$x_{11}$};

\foreach \from/\to in {x1/x2,x3/x4,x5/x6,x7/x8,x9/x10,x11/x0}
    \draw (\from) -- (\to);

\foreach \from in {x0,x1}
    \foreach \to in {x4,x5,x8,x9}
        \draw (\from) -- (\to);

\foreach \from in {x4,x5}
    \foreach \to in {x0,x1,x8,x9}
        \draw (\from) -- (\to);

\foreach \from in {x8,x9}
    \foreach \to in {x0,x1,x4,x5}
        \draw (\from) -- (\to);

\foreach \from in {x2,x3}
    \foreach \to in {x6,x7,x10,x11}
        \draw (\from) -- (\to);

\foreach \from in {x6,x7}
    \foreach \to in {x2,x3,x10,x11}
        \draw (\from) -- (\to);

\foreach \from in {x10,x11}
    \foreach \to in {x2,x3,x6,x7}
        \draw (\from) -- (\to);

%
%
%
%
%
%
\end{tikzpicture}
\end{center}
\caption{The Graph~$G_{4n}$ from Theorem~\ref{thm:gem-P1+2P_2} when $n=3$.}
\label{fig:thm:gem-P1+2P_2}
\end{figure}

Suppose, for contradiction that~$G_{4n}$ contains an induced $P_1+\nobreak 2P_2$.
Without loss of generality, the set~$X$ must contain three of the vertices $v_1,v_2,v_3$ of the $P_1+\nobreak 2P_2$.
Since every component of $P_1+\nobreak 2P_2$ contains at most two vertices, without loss of generality we may assume that~$v_1$ is non-adjacent to both~$v_2$ and~$v_3$.
However, every vertex of~$G_{4n}[X]$ has exactly one non-neighbour in~$X$.
This contradiction shows that~$G_{4n}$ is indeed $(P_1+\nobreak 2P_2)$-free.

Every vertex in~$X$ has exactly one neighbour in~$Y$ and vice versa.
This means that any~$K_3$ in~$G_{4n}$ must lie entirely in~$G_{4n}[X]$ or~$G_{4n}[Y]$.
Since~$G_{4n}[X]$ or~$G_{4n}[Y]$ are both complements of perfect matchings and every vertex of~$\overline{P_1+P_4}$ lies in one of three induced~$K_3$'s, which are pairwise non-disjoint, it follows that~$G_{4n}$ is $\overline{P_1+P_4}$-free.

It remains to show that the graphs~$G_{4n}$ form an infinite antichain with respect to the induced subgraph relation.
Since $n \geq 3$, every vertex in~$X$ (resp.~$Y$) has at least two neighbours in~$X$ (resp.~$Y$) that are pairwise adjacent.
Therefore, given~$x_1$, we can determine which vertices lie in~$X$ and which lie in~$Y$.
Every vertex in~$X$ (resp.~$Y$) has a unique neighbour in~$Y$ (resp.~$X$) and a unique non-neighbour in~$X$ (resp.~$Y$).
Therefore, by specifying which vertex in~$G_{4n}$ is~$x_1$, we uniquely determine $x_2,\ldots,x_{4n}$.
Suppose~$G_{4n}$ an induced subgraph of~$G_{4m}$ for some $m \geq 3$.
Then $n \leq m$ due to the number of vertices.
By symmetry, we may assume that the induced copy of~$G_{4n}$ in~$G_{4m}$ has vertex~$x_1$ of~$G_{4n}$ in the position of vertex~$x_1$ in~$G_{4m}$.
Then the induced copy of~$G_{4n}$ must have vertices $x_2,\ldots,x_{4n}$ in the same position as $x_2,\ldots,x_{4n}$ in~$G_{4m}$, respectively.
Now~$x_1$ and~$x_{4n}$ are non-adjacent in~$G_{4n}$.
If $n < m$ then~$x_1$ and~$x_{4n}$ are adjacent in~$G_{4m}$, a contradiction.
We conclude that if~$G_{4n}$ is an induced subgraph of~$G_{4m}$ then $n=m$.
In other words $\{G_{4n}\; |\; n \geq 3\}$ is an infinite antichain with respect to the induced subgraph relation.
\qed
\end{proof}

\section{State of the Art and Future Work}\label{s-state}

\begin{sloppypar}
In this section we summarise what is currently known about well-quasi-orderability and boundedness of clique-width, taking in to account the results proved in this paper.
We also give a number of directions for future work.
\end{sloppypar}

\medskip
Given four graphs $H_1,H_2,H_3,H_4$, the class of $(H_1,H_2)$-free graphs and the class of $(H_3,H_4)$-free graphs are {\em equivalent} if the unordered pair $H_3,H_4$ can be obtained from the unordered pair $H_1,H_2$ by some combination of the operations (i) complementing both graphs in the pair and (ii) if one of the graphs in the pair is~$K_3$, replacing it with $\overline{P_1+P_3}$ or vice versa.
If two classes are equivalent, then one of them is well-quasi-ordered with respect to the induced subgraph relation if and only if the other one is~\cite{KL11}.
Similarly, if two classes are equivalent, then one of them has bounded clique-width if and only if the other one does~\cite{DP15}.
We use this terminology in the remainder of this section.

\subsection{Well-Quasi-Ordering}
Atminas and Lozin~\cite{AL15} proved that the class of $(K_3,P_6)$-free graphs is well-quasi-ordered by the induced subgraph relation, while the class of $(\overline{2P_1+P_2},P_6)$-free graphs is not.
Updating the classification in~\cite{KL11} with these two results and the three results proved in this paper (Theorems~\ref{thm:diamond-p2+p3-wqo}--\ref{thm:gem-P1+2P_2})
leads to the following classification:
\begin{theorem}\label{t-wqowqo}
Let~${\cal G}$ be a class of graphs defined by two forbidden induced subgraphs. Then:
\begin{enumerate}
\item ${\cal G}$ is well-quasi-ordered with respect to the {\em labelled} induced subgraph relation if it is equivalent
to a class of $(H_1,H_2)$-free graphs such that one of the following holds:
\begin{enumerate}[(i)]
\item $H_1$ or $H_2 \ssi P_4$;
\item $H_1=sP_1$ and $H_2=K_t$ for some $s,t$;
\item $H_1 \ssi \overline{3P_1}$ and $H_2 \ssi 2P_1+P_3$ or~$P_6$;
\item $H_1 \ssi \overline{2P_1+P_2}$ and $H_2 \ssi P_2+P_3$ or~$P_5$.
\end{enumerate}
\item ${\cal G}$ is not well-quasi-ordered with respect to the induced subgraph relation if it is equivalent to a class of $(H_1,H_2)$-free graphs such that one of the following holds: 
\begin{enumerate}[(i)]
\item neither~$H_1$ nor~$H_2$ is a linear forest (disjoint union of paths);
\item $H_1 \si \overline{3P_1}$ and $H_2 \si 3P_1+P_2, 3P_2$ or~$2P_3$;
\item $H_1 \si \overline{2P_2}$ and $H_2 \si 4P_1$ or~$2P_2$;
\item $H_1 \si \overline{2P_1+P_2}$ and $H_2 \si 4P_1, P_2+P_4$ or~$P_6$;
\item $H_1 \si \overline{P_1+P_4}$ and $H_2 \si P_1+2P_2$.
\end{enumerate}
\end{enumerate}
\end{theorem}
Note that in Theorem~\ref{t-wqowqo} every class that is well-quasi-ordered with respect to the induced subgraph relation is also well-quasi-ordered with respect to the {\em labelled} induced subgraph relation (see~\cite{AL15,KL11} and Theorem~\ref{thm:diamond-p2+p3-wqo}). 
This agrees with a conjecture of Atminas and Lozin~\cite{AL15} stating that these concepts coincide for hereditary classes~$X$ precisely when~$X$ is defined by a finite collection of forbidden induced subgraphs.
Theorem~\ref{t-wqowqo} leaves us with nine open cases.

\begin{oproblem}\label{o-wqo}
Is the class of $(H_1,H_2)$-free graphs well-quasi-ordered 
by the induced subgraph relation 
when:
\begin{enumerate}[(i)]
\item $H_1=\overline{3P_1}$ and $H_2 \in \{P_1+2P_2, P_1+P_5, P_2+P_4\}$;
\item $H_1=\overline{2P_1+P_2}$ and $H_2 \in \{P_1+2P_2, P_1+P_4\}$;
\item $H_1=\overline{P_1+P_4}$ and $H_2 \in \{P_1+P_4, 2P_2, P_2+P_3, P_5\}$.
\end{enumerate}
\end{oproblem}

In relation to Open Problem~\ref{o-wqo}, we mention that the infinite antichain for $(\overline{P_1+P_4},P_1+\nobreak 2P_2)$-free graphs
was initially found by a computer search. This computer search also showed that similar antichains do not exist for any of the remaining nine open cases. As such, constructing antichains for these cases is likely to be a challenging problem and this suggests that many of these classes may in fact be well-quasi-ordered.
Some of these remaining classes have been shown to have bounded clique-width~\cite{BK05,BLM04b,BLM04,DDP15}. We believe that some of the structural characterizations for proving these results may be useful for showing well-quasi-orderability.
Indeed, we are currently trying to prove that the classes of $(K_3,P_1+\nobreak P_5)$-free graphs and $(K_3,P_2+\nobreak P_4)$-free graphs are well-quasi-ordered via 
a special type of graph partition that avoids rainbow triangles and rainbow anti-triangles.

\subsection{Clique-Width}

The following theorem from~\cite{DDP15} describes exactly for which pairs $(H_1,H_2)$ the (un)boundedness of the clique-width of $(H_1,H_2)$-free
graphs has been determined.

\begin{theorem}\label{thm:classification2}
Let~${\cal G}$ be a class of graphs defined by two forbidden induced subgraphs. Then:
\begin{enumerate}
\item ${\cal G}$ has bounded clique-width if it is equivalent
to a class of $(H_1,H_2)$-free graphs such that one of the following holds:
\begin{enumerate}[(i)]
\item \label{thm:classification2:bdd:P4} $H_1$ or $H_2 \ssi P_4$;
\item \label{thm:classification2:bdd:ramsey} $H_1=sP_1$ and $H_2=K_t$ for some $s,t$;
\item \label{thm:classification2:bdd:P_1+P_3} $H_1 \ssi P_1+\nobreak P_3$ and $\overline{H_2} \ssi K_{1,3}+\nobreak 3P_1,\; K_{1,3}+\nobreak P_2,\;\allowbreak P_1+\nobreak P_2+\nobreak P_3,\;\allowbreak P_1+\nobreak P_5,\;\allowbreak P_1+\nobreak S_{1,1,2},\;\allowbreak P_6,\; \allowbreak S_{1,1,3}$ or~$S_{1,2,2}$;
\item \label{thm:classification2:bdd:2P_1+P_2} $H_1 \ssi 2P_1+\nobreak P_2$ and $\overline{H_2}\ssi P_1+\nobreak 2P_2,\; 2P_1+\nobreak P_3,\; 3P_1+\nobreak P_2$ or~$P_2+\nobreak P_3$;
\item \label{thm:classification2:bdd:P_1+P_4} $H_1 \subseteq_i P_1+\nobreak P_4$ and $\overline{H_2} \ssi P_1+\nobreak P_4$ or~$P_5$;
\item \label{thm:classification2:bdd:4P_1} $H_1 \subseteq_i 4P_1$ and $\overline{H_2} \ssi 2P_1+\nobreak P_3$;
\item \label{thm:classification2:bdd:K_13} $H_1,\overline{H_2} \ssi K_{1,3}$.
\end{enumerate}
\item ${\cal G}$ has unbounded clique-width if it is equivalent to a class of $(H_1,H_2)$-free graphs such that one of the following holds:
\begin{enumerate}[(i)]
\item \label{thm:classification2:unbdd:not-in-S} $H_1\not\in {\cal S}$ and $H_2 \not \in {\cal S}$;
\item \label{thm:classification2:unbdd:not-in-co-S} $\overline{H_1}\notin {\cal S}$ and $\overline{H_2} \not \in {\cal S}$;
\item \label{thm:classification2:unbdd:K_13or2P_2} $H_1 \si K_{1,3}$ or~$2P_2$ and $\overline{H_2} \si 4P_1$ or~$2P_2$;
\item \label{thm:classification2:unbdd:2P_1+P_2} $H_1 \si 2P_1+\nobreak P_2$ and $\overline{H_2} \si K_{1,3},\; 5P_1,\; P_2+\nobreak P_4$ or~$P_6$;
\item \label{thm:classification2:unbdd:3P_1} $H_1 \si 3P_1$ and $\overline{H_2} \si 2P_1+\nobreak 2P_2,\; 2P_1+\nobreak P_4,\; 4P_1+\nobreak P_2,\; 3P_2$ or~$2P_3$;
\item \label{thm:classification2:unbdd:4P_1} $H_1 \si 4P_1$ and $\overline{H_2} \si P_1 +\nobreak P_4$ or~$3P_1+\nobreak P_2$.
\end{enumerate}
\end{enumerate}
\end{theorem}
This leaves us with the following  eight non-equivalent open cases.

\begin{oproblem}\label{oprob:twographs}
Does the class of $(H_1,H_2)$-free graphs have bounded or unbounded clique-width when:
\begin{enumerate}[(i)]
\item \label{oprob:twographs:3P_1} $H_1=3P_1$ and $\overline{H_2} \in \{P_1+\nobreak S_{1,1,3},\allowbreak P_2+\nobreak P_4,\allowbreak S_{1,2,3}\}$;
\item\label{oprob:twographs:2P_1+P_2} $H_1=2P_1+\nobreak P_2$ and $\overline{H_2} \in \{P_1+\nobreak P_2+\nobreak P_3,\allowbreak P_1+\nobreak P_5\}$;
\item \label{oprob:twographs:P_1+P_4} $H_1=P_1+\nobreak P_4$ and $\overline{H_2} \in \{P_1+\nobreak 2P_2,\allowbreak P_2+\nobreak P_3\}$ or
\item \label{oprob:twographs:2P_1+P_3} $H_1=\overline{H_2}=2P_1+\nobreak P_3$.
\end{enumerate}
\end{oproblem}

A potential direction for future research related to determining boundedness of clique-width is investigating linear clique-width for classes defined by two forbidden induced subgraphs.
Indeed, it is not hard to show that $k$-uniform graphs have bounded linear clique-width.
Again, we can use complementations and vertex deletions when dealing with this parameter.

\subsection{Well-Quasi-Ordering versus Clique-Width}

We recall that all known results for bigenic classes of graphs verify Conjecture~\ref{c-f}. Note that Conjecture~\ref{c-f} is verified directly if a graph class has bounded clique-width.
Brandst{\"a}dt, Le and Mosca~\cite{BLM04b} proved that the class of $(\overline{P_1+P_4},P_1+\nobreak P_4)$-free graphs (and thus the class of $(\overline{2P_1+P_2},P_1+\nobreak P_4)$-free graphs) has bounded clique-width.
Using results from~\cite{BK05}, the same authors proved in~\cite{BLM04} that the class of $(\overline{P_1+P_4},P_5)$-free graphs (and thus the class of $(\overline{2P_2},P_1+\nobreak P_4)$-free graphs)
has bounded clique-width.
Dabrowski, Dross and Paulusma~\cite{DDP15} showed that the classes of $(K_3,P_1+\nobreak 2P_2)$-free graphs, $(K_3,P_1+\nobreak P_5)$-free graphs and 
$(\overline{2P_1+P_2},P_1+\nobreak 2P_2)$-free graphs all have bounded clique-width. 
Hence, there are only two (non-equivalent) bigenic classes of graphs left for which Conjecture~\ref{c-f} needs to be verified.
We state these two open cases below.

\begin{oproblem}\label{o-con}
Is Conjecture~\ref{c-f} true for the class of $(H_1,H_2)$-free graphs when:
\begin{enumerate}[(i)]
\item $H_1=K_3$ and $H_2=P_2+\nobreak P_4$ or when
\item $H_1=\overline{P_1+P_4}$ and $H_2=P_2+\nobreak P_3$?
\end{enumerate}
\end{oproblem}

\noindent
As can be seen from Open Problems~\ref{o-wqo} and~\ref{oprob:twographs}, for both classes we know neither whether they are well-quasi-ordered by the induced subgraph relation nor whether their clique-width is bounded.

\bibliographystyle{abbrv}
\bibliography{mybib}

\begin{thebibliography}{10}

\bibitem{AL15}
A.~Atminas and V.~V. Lozin.
\newblock Labelled induced subgraphs and well-quasi-ordering.
\newblock {\em Order}, 32(3):313--328, 2015.

\bibitem{Bo96}
H.~L. Bodlaender.
\newblock A linear-time algorithm for finding tree-decompositions of small
  treewidth.
\newblock {\em SIAM Journal on Computing}, 25(6):1305--1317, 1996.

\bibitem{BK05}
A.~Brandst{\"a}dt and D.~Kratsch.
\newblock On the structure of ({$P_5$},gem)-free graphs.
\newblock {\em Discrete Applied Mathematics}, 145(2):155--166, 2005.

\bibitem{BLM04b}
A.~Brandst{\"a}dt, H.-O. Le, and R.~Mosca.
\newblock Gem- and co-gem-free graphs have bounded clique-width.
\newblock {\em International Journal of Foundations of Computer Science},
  15(1):163--185, 2004.

\bibitem{BLM04}
A.~Brandst{\"a}dt, H.-O. Le, and R.~Mosca.
\newblock Chordal co-gem-free and ({$P_5$},gem)-free graphs have bounded
  clique-width.
\newblock {\em Discrete Applied Mathematics}, 145(2):232--241, 2005.

\bibitem{Co92}
B.~Courcelle.
\newblock The monadic second-order logic of graphs {III:} tree-decompositions,
  minor and complexity issues.
\newblock {\em Informatique Th\'eorique et Applications}, 26(3):257--286, 1992.

\bibitem{Co14}
B.~Courcelle.
\newblock Clique-width and edge contraction.
\newblock {\em Information Processing Letters}, 114(1--2):42--44, 2014.

\bibitem{CMR00}
B.~Courcelle, J.~A. Makowsky, and U.~Rotics.
\newblock Linear time solvable optimization problems on graphs of bounded
  clique-width.
\newblock {\em Theory of Computing Systems}, 33(2):125--150, 2000.

\bibitem{CO00}
B.~Courcelle and S.~Olariu.
\newblock Upper bounds to the clique width of graphs.
\newblock {\em Discrete Applied Mathematics}, 101(1--3):77--114, 2000.

\bibitem{DDP15}
K.~K. Dabrowski, F.~Dross, and D.~Paulusma.
\newblock Colouring diamond-free graphs.
\newblock {\em Proc. SWAT 2016, LIPIcs}, 53:16:1--16:14, 2016.

\bibitem{DHP0}
K.~K. Dabrowski, S.~Huang, and D.~Paulusma.
\newblock Bounding clique-width via perfect graphs.
\newblock {\em Journal of Computer and System Sciences}, (in press).

\bibitem{DLP16-conf}
K.~K. Dabrowski, V.~V. Lozin, and D.~Paulusma.
\newblock Well-quasi-ordering versus clique-width: New results on bigenic
  classes.
\newblock {\em Proc. IWOCA 2016, LNCS}, 9843:253--265, 2016.

\bibitem{DP15}
K.~K. Dabrowski and D.~Paulusma.
\newblock Clique-width of graph classes defined by two forbidden induced
  subgraphs.
\newblock {\em The Computer Journal}, 59(5):650--666, 2016.

\bibitem{DRT10}
J.~Daligault, M.~Rao, and S.~Thomass{\'e}.
\newblock Well-quasi-order of relabel functions.
\newblock {\em Order}, 27(3):301--315, 2010.

\bibitem{Da90}
P.~Damaschke.
\newblock Induced subgraphs and well-quasi-ordering.
\newblock {\em Journal of Graph Theory}, 14(4):427--435, 1990.

\bibitem{EGW01}
W.~Espelage, F.~Gurski, and E.~Wanke.
\newblock How to solve {NP-hard} graph problems on clique-width bounded graphs
  in polynomial time.
\newblock {\em Proc. WG 2001, LNCS}, 2204:117--128, 2001.

\bibitem{FS01}
A.~Finkel and P.~Schnoebelen.
\newblock Well-structured transition systems everywhere!
\newblock {\em Theoretical Computer Science}, 256(1--2):63--92, 2001.

\bibitem{Higman52}
G.~Higman.
\newblock Ordering by divisibility in abstract algebras.
\newblock {\em Proceedings of the London Mathematical Society},
  s3--2(1):326--336, 1952.

\bibitem{KR03b}
D.~Kobler and U.~Rotics.
\newblock Edge dominating set and colorings on graphs with fixed clique-width.
\newblock {\em Discrete Applied Mathematics}, 126(2--3):197--221, 2003.

\bibitem{KL11}
N.~Korpelainen and V.~V. Lozin.
\newblock Two forbidden induced subgraphs and well-quasi-ordering.
\newblock {\em Discrete Mathematics}, 311(16):1813--1822, 2011.

\bibitem{Kruskal72}
J.~B. Kruskal.
\newblock The theory of well-quasi-ordering: A frequently discovered concept.
\newblock {\em Journal of Combinatorial Theory, Series A}, 13(3):297--305,
  1972.

\bibitem{LRZ15}
V.~V. Lozin, I.~Razgon, and V.~Zamaraev.
\newblock Well-quasi-ordering does not imply bounded clique-width.
\newblock {\em Proc. WG 2015, LNCS}, 9224:351--359, 2016.

\bibitem{OS06}
S.-I. Oum and P.~D. Seymour.
\newblock Approximating clique-width and branch-width.
\newblock {\em Journal of Combinatorial Theory, Series B}, 96(4):514--528,
  2006.

\bibitem{Ra07}
M.~Rao.
\newblock {MSOL} partitioning problems on graphs of bounded treewidth and
  clique-width.
\newblock {\em Theoretical Computer Science}, 377(1--3):260--267, 2007.

\bibitem{RS90}
N.~Robertson and P.~D. Seymour.
\newblock Graph minors. {IV.} {Tree-width} and well-quasi-ordering.
\newblock {\em Journal of Combinatorial Theory, Series B}, 48(2):227--254,
  1990.

\bibitem{RS04-Wagner}
N.~Robertson and P.~D. Seymour.
\newblock Graph minors. {XX.} {Wagner's} conjecture.
\newblock {\em Journal of Combinatorial Theory, Series B}, 92(2):325--357,
  2004.

\end{thebibliography}

\newpage
\appendix

\section{Proofs of the Claims in Lemma~\ref{lem:diamond-P_2+P_3-free-C5-non-free} from~\cite[Lemma~10]{DHP0}}\label{app:c5}
\setcounter{ctrclaim}{0}
\clm{\em $G[V_i \cup X]$ is a matching.}
Indeed, if some vertex~$x$ in~$V_i$
(respectively~$X$) is adjacent to two vertices $y_1,y_2$ in~$X$
(respectively~$V_i$), then $G[v_{i+2},v_{i+3},y_1,x,y_2]$ is a $P_2+P_3$.

\medskip
\clm{\em If~$V_i$ and~$V_j$ are opposite, then $G[V_i\cup V_j]$ is a matching.}
Suppose for contradiction that $x \in V_1$ is adjacent to two vertices $y,y'
\in V_3$. Then $G[v_2,x,y,y']$ would be a $\overline{2P_1+P_2}$, a
contradiction.

\medskip
\clm{\em If~$V_i$ and~$V_j$ are consecutive, then $G[V_i \cup V_j]$ is a
co-matching.} Suppose for contradiction that $x \in V_1$ is non-adjacent to two
vertices $y,y' \in V_2$. Then $G[x,v_5,y,v_3,y']$ is a $P_2+P_3$, a
contradiction.

\medskip
\clm{\em If~$V_i$ is large, then~$X$ is anti-complete to $V_{i-2} \cup
V_{i+2}$.} Suppose for contradiction that~$V_3$ is large and $x \in X$ has a
neighbour $y \in V_1$. Then since~$V_3$ is large and both $G[X \cup V_3]$ and
$G[V_1 \cup V_3]$ are matchings, there must be a vertex $z \in V_3$ that is
non-adjacent to both~$x$ and~$y$. Then $G[x,y,v_3,v_4,z]$ is a $P_2+P_3$, a
contradiction.

\medskip
\clm{\em If~$V_i$ is large, then~$V_{i-1}$ is anti-complete to~$V_{i+1}$.}
Suppose for contradiction that~$V_2$ is large and $x \in V_1$ has a neighbour
$y \in V_3$. Since~$V_2$ is large and each vertex in $V_1 \cup V_3$ has at most
one non-neighbour in~$V_2$, there must be a vertex $z \in V_2$ that is adjacent
to both~$x$ and~$y$. Now $G[x,y,v_2,z]$ is a $\overline{2P_1+P_2}$, a
contradiction.

\medskip
\clm{\em If $V_{i-1},V_{i},V_{i+1}$ are large, then~$V_i$ is complete to
$V_{i-1} \cup V_{i+1}$.} Suppose for contradiction that $V_1,V_2,V_3$ are large
and some vertex $x \in V_1$ is non-adjacent to a vertex $y \in V_2$.
Since~$V_3$ is large and $G[V_2 \cup V_3]$ is a co-matching, there must be two
vertices $z,z' \in V_3$, adjacent to~$y$. By the previous claim, since~$V_2$ is
large, $z,z'$ must be non-adjacent to~$x$. Therefore $G[x,v_5,z,y,z']$ is a
$P_2+\nobreak P_3$, which is a contradiction.

\section{Proofs of the Claims in Lemma~\ref{lem:diamond-P_2+P_3-free-C4-non-free} from~\cite[Lemma~11]{DHP0}}\label{app:c4-1}

\setcounter{ctrclaim}{0}
\clm{\em $V_i$ is independent for $i\in \{1,2\}$.}
If $x,y \in V_i$ were adjacent then $G[x,y,v_{i+1},v_{i+3}]$ would be a $\overline{2P_1+P_2}$.

\medskip
\clm{\em $W_i$ is independent for $i\in \{1,2,3,4\}$.}
If $x,y \in W_i$ were adjacent then $G[x,y,v_{i+1},v_{i+2},v_{i+3}]$ would be a $P_2+\nobreak P_3$.

\medskip
\clm{\em $X$ is independent.}
If $x,y \in X$ were adjacent then $G[x,y,v_1,v_2,v_3]$ would be a $P_2+\nobreak P_3$.

\medskip
\clm{\em $W_i$ is anti-complete to~$X$ for $i\in \{1,2,3,4\}$.}
If $x \in X$ were adjacent to $y \in W_i$ then $G[x,y,v_{i+1},v_{i+2},v_{i+3}]$ would be a $P_2+\nobreak P_3$.

\medskip
\clm{\em For $i \in \{1,2\}$ either~$W_i$ or~$W_{i+2}$ is empty. Therefore, we may assume by symmetry that $W_3=\emptyset$ and $W_4=\emptyset$.}
To show this, first suppose that vertices $x\in W_1$ and $y\in W_3$ are adjacent. 
Then $G[v_1,v_2,v_3,y,x]$ is a~$C_5$, which is a contradiction.
Therefore,  $W_1$ is anti-complete to~$W_3$.
If both~$W_1$ and~$W_3$ are non-empty then by our earlier assumption they must each contain at least two vertices.
Suppose that $x\in W_1$ and $y,z\in W_3$.
In this case $G[x,v_1,y,v_3,z]$ is a $P_2+\nobreak P_3$, a contradiction.
We conclude that at least one of~$W_1$ and~$W_3$ must be empty.
Without loss of generality, we assume that~$W_3$ is empty. Similarly, we assume that~$W_4$ is empty. 
\end{document}